\documentclass[11pt]{amsart}

\usepackage{amsthm}
\usepackage{amssymb}
\usepackage{amsrefs}
\usepackage{enumitem}
\usepackage{graphicx}
\usepackage{todonotes}
\usepackage{comment}

\usepackage[colorlinks=true]{hyperref}

\usepackage[justification=centering]{caption}

\theoremstyle{plain}
\newtheorem{theorem}{Theorem}[section]
\newtheorem{proposition}[theorem]{Proposition}
\newtheorem{lemma}[theorem]{Lemma}
\newtheorem{corollary}[theorem]{Corollary}
\newtheorem{claim}{}[theorem]

\theoremstyle{definition}
\newtheorem{definition}[theorem]{Definition}

\newcommand{\cl}{\operatorname{cl}}
\newcommand{\tw}{\operatorname{tw}}
\newcommand{\bw}{\operatorname{bw}}
\newcommand{\dash}{\nobreakdash-\hspace{0mm}}

\title[Supersolvable and saturated matroids]{Supersolvable saturated matroids and chordal graphs}

\author[Mayhew]{Dillon Mayhew}

\author[Probert]{Andrew Probert}

\begin{document}

\begin{abstract}
A matroid is supersolvable if it has a maximal chain of flats each of which is modular.
A matroid is saturated if every round flat is modular.
In this article we present supersolvable saturated matroids as analogues to chordal graphs, and we show that several results for chordal graphs hold in this matroid context.
In particular, we consider matroid analogues of the reduced clique graph and clique trees for chordal graphs.
The latter is a maximum-weight spanning tree of the former.
We also show that the matroid analogue of a clique tree is an optimal decomposition for the matroid parameter of tree-width.
\end{abstract}

\maketitle

\section{Introduction}

The study of chordal graphs is well established, and dates to work by Dirac \cite{Dirac} and Berge \cite{Berge}.
Our contribution here is to consider a new analogue of chordality for matroids.
A graph is chordal if every cycle with at least four vertices has a chord.
This leads fairly directly to the definition of a chordal matroid used by Cordovil, Forge, and Klein \cite{CFK04}.
If $C$ is a circuit in a matroid, then a \emph{chord} of $C$ is an element $z\notin C$ such that there is a partition of $C$ into parts $A$ and $B$ where $A\cup z$ and $B\cup z$ are both circuits.
We will say that a matroid is \emph{$C$\dash chordal} if every circuit with size at least four has a chord.
(Cordovil et al.\ call such a matroid chordal, but we will try to avoid confusion by reserving that term solely for graphs.)

In this article we concentrate on a different matroid analogue for chordality.
An alternative characterisation of chordal graphs is due to Dirac \cite{Dirac}: a vertex is \emph{simplicial} if its neighbours are pairwise adjacent.
Now $G$ is chordal if and only if it has a simplicial vertex $v$ such that $G-v$ is chordal.
This definition is well suited for matroid purposes, because the edges not incident with a simplicial vertex comprise a modular hyperplane in the corresponding graphic matroid.
(A flat $F$ is \emph{modular} if $r(F)+r(F') = r(F\cap F')+r(F\cup F')$ for every flat $F'$.
A hyperplane is modular if and only if it has a non-empty intersection with every rank-two flat of the matroid.)
Now we can recursively consider the class of matroids $\mathcal{M}$ such that $M$ is in $\mathcal{M}$ if and only if $M$ has a modular hyperplane $H$ where restricting $M$ to $H$ produces a matroid in $\mathcal{M}$.
The class $\mathcal{M}$ is exactly the family of \emph{supersolvable} matroids, introduced by Stanley \cite{stanley}.

\begin{figure}[htb]
    \centering
    \includegraphics{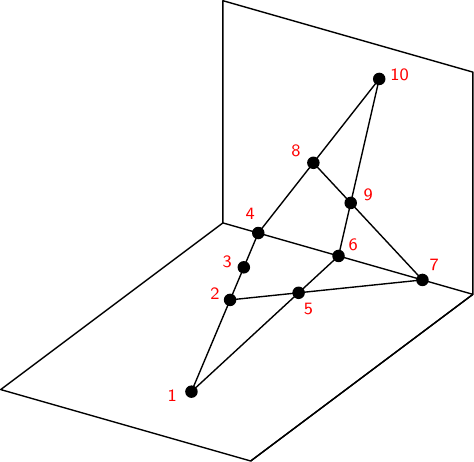}
    \caption{A supersolvable matroid}
    \label{Fig2}
\end{figure}

Figure \ref{Fig2} shows a geometric representation of a rank-four matroid, $M$.
We see that the hyperplane $F_{3}=\{1,2,3,4,5,6,7\}$ is modular, since every rank-two flat has a non-empty intersection with $F_{3}$.
In the same way, $F_{2}=\{1,2,3,4\}$ is a modular hyperplane of the restriction to $F_{3}$, and $F_{1}=\{1\}$ is a modular hyperplane of the restriction to $F_{2}$.
Finally, $\emptyset$ is a modular hyperplane of the restriction to $F_{1}$.
It follows that $M$ is supersolvable.

It turns out that the condition of supersolvability is not strong enough for our purposes because supersolvable matroids may fail to have properties shared by all graphic matroids.
To expand on this point, we consider matroid analogues of cliques in a graph.
Let $F$ be a flat of a matroid.
Then $F$ is \emph{round} if there is no pair of flats $(F_{1},F_{2})$ such that $F=F_{1}\cup F_{2}$ and $F_{1}$ and $F_{2}$ are properly contained in $F$.
Let $G$ be a graph and let $F$ be a flat of the graphic matroid $M(G)$.
Then $F$ is round if and only if $G[F]$ is a clique (Proposition \ref{prop1}).
Therefore we think of round flats as the matroid analogues of cliques.
In graphic matroids every round flat is modular but this is not true for matroids in general, nor is it true for supersolvable matroids.
For example, if $M$ is the matroid in Figure \ref{Fig2}, then $\{4,6,7,8,9,10\}$ is a round hyperplane, since it cannot be expressed as the union of two flats that it properly contains.
However, it is not modular, since it has an empty intersection with the rank-two flat $\{3,5\}$.

We define a matroid to be \emph{saturated} if every round flat is modular.
Thus saturated matroids can be thought of as analogues to graphs.
To this condition, we add the condition of supersolvability to obtain our matroid analogue of chordal graphs.
So our fundamental objects of study are supersolvable and saturated matroids.
The graphic matroid $M(G)$ is supersolvable and saturated if and only $G$ is chordal (Corollary \ref{graphic-implies-saturation} and Proposition \ref{chordal-iff-supersolv}).
Many other examples arise: for example, the matroids that are constructed using generalised parallel connections, starting with the projective geometries of a given order.
Any such matroid is supersolvable and saturated.

The class of supersolvable saturated matroids is properly contained in the class of $C$\dash chordal matroids (Proposition \ref{SSS-implies-chordal}).
So our focus is on a proper subclass of $C$\dash chordal matroids.
The relationships between the conditions of supersovability, saturation, and $C$\dash chordality are illustrated in Figure \ref{Fig4}.
We will justify this Venn diagram in Section \ref{CHORDALITY}.
\begin{figure}[htb]
    \centering
    \includegraphics[scale=1.1]{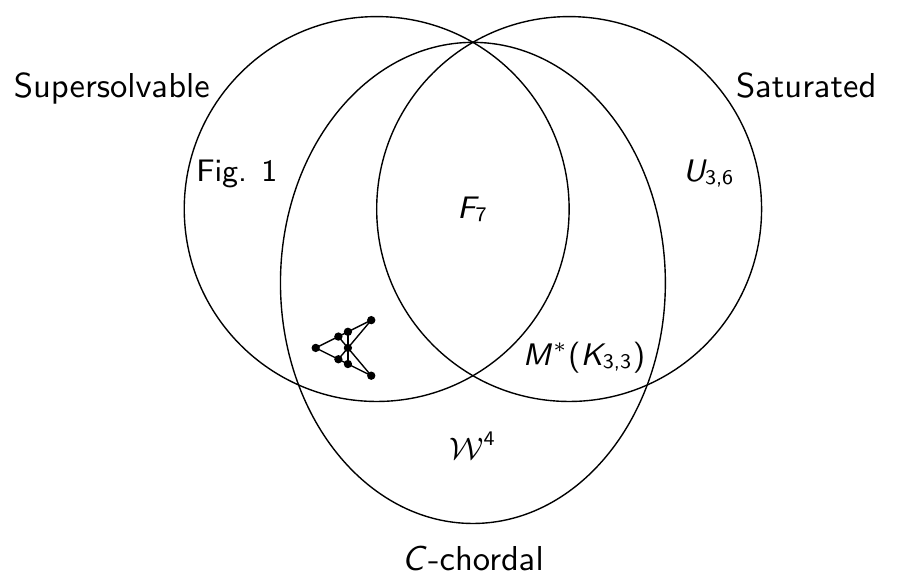}
    \caption{Three matroid definitions}
    \label{Fig4}
\end{figure}

Our main focus is showing that many facts about chordal graphs have analogues in the class of supersolvable saturated matroids.
In particular, Section \ref{RC-and-ROTgraphs} introduces one of our main ideas: the \emph{rotunda graph} of such a matroid.
A \emph{rotunda} is a maximal round flat.
The vertices of the rotunda graph are the rotunda of the matroid.
Assume that $R_{1}$ and $R_{2}$ are distinct rotunda with a non-empty intersection and that $(F_{1},F_{2})$ is a pair of modular flats of $M$ such that $E(M)=F_{1}\cup F_{2}$ and neither $F_{1}$ nor $F_{2}$ is equal to $E(M)$.
If $R_{i}\subseteq F_{i}$ for $i=1,2$ and $F_{1}\cap F_{2}=R_{1}\cap R_{2}$, then we make $R_{1}$ and $R_{2}$ adjacent in the rotunda graph.
The idea of a rotunda graph is analogous to the \emph{reduced clique graph} introduced by Galinier, Habib, and Paul in \cite{GHP95} (where it is called a clique graph).
If $G$ is a chordal graph, then the vertices of the reduced clique graph of $G$ are the maximal cliques of $G$.
If $C$ and $C'$ are maximal cliques then they are adjacent if $C\cap C'\ne \emptyset$ and any path from a vertex of $C-C'$ to a vertex of $C'-C$ uses a vertex of $C\cap C'$.

If $G$ is a chordal graph then the reduced clique graph of $G$ and the rotunda graph of $M(G)$ need not be the same, but this is only because $G$ may have low connectivity.
In Proposition \ref{2-con-same} we show that when $G$ is $2$\dash connected the reduced clique graph of $G$ and the rotunda graph of $M(G)$ are identical.
We can go further than this: the class of reduced clique graphs and the class of rotunda graphs are identical.

\begin{theorem}
\label{ROTgraph-same-as-RCgraph}
Let $H$ be a graph.
Then $H$ is isomorphic to the rotunda graph of a supersolvable saturated matroid if and only if $H$ is isomorphic to the reduced clique graph of a chordal graph.
\end{theorem}

We prove this theorem in Section \ref{VS}.
It tells us that although a supersolvable saturated matroid may be far from graphic, the structure of its rotunda will be mirrored by the structure of maximal cliques in a chordal graph.

Knowing that these two classes of graphs are identical allows us to deduce facts about the structure of rotunda graphs from the facts about reduced clique graphs that we list in \cite{graphpaper}.
For example, in \cite{graphpaper} we show that the reduced clique graph of a chordal graph may have induced cycles of length three, four, or six, but not five.
Therefore the same statement applies to rotunda graphs.
We conjecture that a reduced clique graph cannot have an induced cycle of length greater than six, so we therefore conjecture that the same statement holds for rotunda graphs.
In \cite{graphpaper} we show that no rotunda graph can be isomorphic to a cycle of length at least four.
Thus the class of rotunda graphs is properly contained in the class of graphs with no induced cycle of length five.
We also believe that every chordal graph is isomorphic to the rotunda graph of some supersolvable saturated matroid, and that
there is a polynomial-time algorithm for recognising when a given graph is isomorphic to some rotunda graph.

A \emph{clique tree} of the graph $G$ is a tree whose nodes are the maximal cliques of $G$, where the set of maximal cliques containing an arbitrary vertex $v\in V(G)$ induces a subtree of $T$.
Clique trees were introduced by Gavril \cite{Gavril}, who showed that $G$ has a clique tree if and only if $G$ is chordal.
The analogue for a supersolvable saturated matroid $M$ is a \emph{rotunda tree}.
In this case the nodes of the rotunda tree are the rotunda of $M$, and the rotunda containing an arbitrary element $x\in E(M)$ induces a subtree.
A matroid may have a rotunda tree without being supersolvable and saturated.
For example, the matroid in Figure \ref{Fig2} is not saturated, but it does have a rotunda tree (having two nodes, corresponding to $\{1,2,3,4,5,6,7\}$ and $\{4,6,7,8,9,10\}$).

Galinier et al.\ \cite{GHP95} weight the edges of reduced clique graphs.
The edge that joins maximal cliques $C$ and $C'$ is weighted with $|C\cap C'|$.
They then prove that a spanning tree of the reduced clique graph is a clique tree if and only if it has maximum total weight amongst all spanning trees.
(Their proof contains a flaw, which we explain and correct in \cite{graphpaper}.)
In our analogous result we weight the edges of rotunda graphs.
The edge that joins rotunda $R$ and $R'$ is weighted with the rank of $R\cap R'$.
(Our techniques are general enough that we could also weight it with $|R\cap R'|$).
In Section \ref{TREES} we prove the following.

\begin{theorem}
\label{trees_main_1}
Let $M$ be a connected supersolvable and saturated matroid.
Every rotunda tree of $M$ is a spanning tree of the rotunda graph of $M$.
Every edge of the rotunda graph is contained in a rotunda tree.
Moreover, a spanning tree is a rotunda tree if and only if it has maximum weight amongst all spanning trees.
\end{theorem}

In Section \ref{DECOMPS} we concentrate on tree-decompositions of optimal width.
In unpublished work, Heggernes \cite{Heggernes} observed that a clique tree of a chordal graph is an optimal decomposition of the graph with respect to the parameter of tree-width.
A matroid analogue of tree-width was developed by Hlin\v{e}n\'{y} and Whittle \cite{hlineny_whittle}, and in Theorem \ref{tree_theorem_two} we prove the matroid analogue of Heggernes's observation: any rotunda tree of a supersolvable and saturated matroid is an optimal decomposition with respect to the matroid parameter of tree-width.

We refer to \cite{Oxley11} for the foundations of matroid theory.

\section{Preliminaries}

\subsection{Chordal graphs}

Let $G$ be a graph.
If $X$ is a set of vertices in $G$, then $G[X]$ is the subgraph induced by $X$.
We say that a path $P$ is \emph{$X$\dash avoiding} if any vertex of $X$ in $P$ is a terminal vertex of $P$.
A \emph{clique} of $G$ is a set of pairwise adjacent vertices.
We blur the distinction between a subgraph, its vertex set, and its edge set.
So for example we may refer to a clique of the graph $G$ as being a flat in the cyclic matroid $M(G)$.

If $C$ is a cycle of a graph, then a \emph{chord} is an edge that joins two distinct vertices of the cycle without being an edge of the cycle or parallel to any such edge.
A graph is \emph{chordal} if every cycle with at least four vertices has a chord.
Thus a graph is chordal if and only if has no induced cycle with more than three vertices.
Clearly every induced subgraph of a chordal graph is chordal.

Let $G$ be a graph, and let $v$ be a vertex of $G$.
If deleting $v$ from $G$ produces a graph with more connected components than $G$, then $v$ is a \emph{cut-vertex} of $G$.
A connected graph with no cut-vertex is \emph{$2$\dash connected}.

An ordering $v_{1},\ldots, v_{n}$ of the vertices in a graph is a \emph{perfect elimination order} if the neighbours of $v_{i}$ amongst $v_{i+1},\ldots, v_{n}$ form a clique, for each $i$.
A proof of the following can be found in \cite{Golumbic04}*{Theorem 4.1}.

\begin{proposition}
\label{chordal-characterisation}
A graph is chordal if and only if it has a perfect elimination order.
\end{proposition}

\subsection{Modularity}

Let $M$ be a matroid.
The flat $F$ is \emph{modular} if $r(F)+r(F')=r(F\cup F') +r(F\cap F')$ whenever $F'$ is a flat.
Note that the entire ground set is trivially a modular flat.
We also see that the unique rank-zero flat is modular.
The following is proved in \cite{Oxley11}*{Proposition 6.9.2}.

\begin{proposition}
\label{mod-flat-characterisation}
Let $F$ be a flat of the matroid $M$.
Then $F$ is modular if and only if $r(F) + r(F')=r(F\cup F')$ whenever $F'$ is a flat such that $F\cap F'=\emptyset$.
\end{proposition}

It follows easily that if $F$ is a hyperplane, then $F$ is modular if and only if $r(F\cap L) = 1$ whenever $L$ is a rank-$2$ flat not contained in $F$.
We often use an equivalent definition.

\begin{proposition}
\label{mod-flat-new-def}
Let $F$ be a flat of the matroid $M$.
Then $F$ is modular if and only if there is no circuit $C\subseteq F\cup F'$ containing elements from both $F$ and $F'$, whenever $F'$ is a flat that is disjoint from $F$.
\end{proposition}

\begin{proof}
Let $F'$ be an arbitrary flat that is disjoint from $F$.
There is no circuit of $M|(F\cup F')$ that contains elements of both $F$ and $F'$ if and only if $r(F)+r(F') = r(F\cup F')$  \cite{Oxley11}*{Proposition 4.2.1}.
Now the result follows by Proposition \ref{mod-flat-characterisation}.
\end{proof}

The next result combines Proposition 6.9.5 and Corollary 6.9.8 from \cite{Oxley11}.

\begin{proposition}
\label{modular-intersection}
Let $F$ and $F'$ be modular flats of the matroid $M$.
Then  $F\cap F'$ is a modular flat of $M$.
If $F\subseteq X\subseteq E(M)$ then $F$ is a modular flat of $M|X$.
\end{proposition}

\begin{proposition}
\label{short-circuit}
Let $F$ be a modular flat of the matroid $M$ and let $C$ be a circuit of $M$ such that $C\cap F$ is non-empty.
Then $\cl(C-F)\cap F$ is non-empty.
\end{proposition}

\begin{proof}
If $\cl(C-F)\cap F=\emptyset$ then Proposition \ref{mod-flat-new-def} is violated, since $\cl(F-C)$ is a flat that is disjoint from $F$, but $C$ is a circuit that contains elements from both $F$ and $\cl(C-X)$.
\end{proof}

Let $H$ be a modular hyperplane of the matroid $M$, and let $C^{*}$ be the complementary cocircuit.
Let $x$ and $y$ be distinct rank-one flats contained in $C^{*}$.
Then $r(H\cap \cl(x\cup y))=1$, because $H$ is modular.
We say that the rank-one flat $H\cap \cl(x\cup y)$ is the \emph{projection} of $x$ and $y$ onto $H$, and we denote this flat with $P_{H}(x,y)$.
If $x$ and $y$ are elements of $C^{*}$ such that $r(\{x,y\})=2$, then we also use $P_{H}(x,y)$ to stand for $P_{H}(\cl(\{x\}),\cl(\{y\}))$.

\begin{proposition}
\label{mod-hyp-sep-extends}
Let $H$ be a modular hyperplane of the matroid $M$.
Let $X$ be a subset of $E(M)-H$ and let $P$ be the union $\cup P_{H}(x,y)$, where $\{x,y\}$ ranges over all pairs of distinct rank-one flats in $X$.
Let $U$ be a subset of $H$ such that $U$ contains $P$.
Then $\cl(U) = \cl(U\cup X)\cap H$.
\end{proposition}

\begin{proof}
Note that $\cl(U)$ is contained in $H$.
Thus it is obvious that $\cl(U)$ is a subset of $\cl(U\cup X)\cap H$.
Let us assume that the containment is proper, and let $z$ be an element that is in $\cl(U\cup X)\cap H$ but not $\cl(U)$.
Thus $z$ is not in $U$.
There is some circuit $C\subseteq U\cup X\cup z$ that contains $z$.
Let us assume that we have chosen $C$ so that $C-H$ is as small as possible.
If $C-H$ is empty, then $C$ certifies that $z$ is in $\cl(U)$, contrary to hypothesis, so $C-H\ne\emptyset$.
If $C-H$ contains a single element $x$, then $C$ certifies that $x$ is in $\cl(H) = H$, which is a contradiction.
Therefore we can choose $x$ and $y$ to be distinct elements of $C-H$.
Let $p$ be an element in $P_{H}(x,y)$.
Thus $p$ is in $P$ and $\{x,y,p\}$ is a circuit.
Note that $z\ne p$, since $z$ is not in $P\subseteq U$.
We perform strong circuit elimination on $C$ and $\{x,y,p\}$ to obtain the circuit $C'\subseteq (C-x)\cup \{p,z\}$ such that $z$ is in $C'$.
Thus $C'$ is a subset of $U\cup X\cup z$, but $C'-H$ is smaller than $C$.
Now our choice of $C$ is contradicted, and this completes the proof.
\end{proof}

\begin{proposition}
\label{mod-hyp-connected}
Let $H$ be a modular hyperplane of the connected matroid $M$.
Then $M|H$ is connected.
\end{proposition}

\begin{proof}
Assume that $M|H$ is not connected, and let $(U,V)$ be a separation of $M|H$.
Because $M$ is connected, there are circuits of $M$ that contain elements from both $U$ and $V$.
Amongst such circuits choose $C$ so that $C-H$ is as small as possible.
Let $u$ be an element in $C\cap U$ and let $v$ be an element from $C\cap V$.
Note that $C-H$ is not empty since $(U,V)$ is a separation of $M|H$.
Furthermore, $C-H$ does not contain a single element, or else that element would be in $\cl(H)=H$.
Therefore we choose distinct elements $x,y\in C-H$.
Let $p$ be an element in $P_{H}(x,y)$, so that $\{x,y,p\}$ is a circuit of $M$.
Because $p$ is in $H$ we can assume without loss of generality that $p$ is in $U$.
We perform strong circuit elimination on $C$ and $\{x,y,p\}$ to obtain a circuit $C'\subseteq (C-x)\cup \{y,p\}$ that contains $v$.
Note that $C'$ contains $p$, or else it is a proper subset of $C$.
Thus $C'$ contains elements from both $U$ and $V$, but $|C'-H|<|C-H|$, and we have a contradiction.
Therefore $M|H$ is connected.
\end{proof}

\subsection{Roundness}

A \emph{proper} flat of a matroid is one that is not equal to the entire ground set.

\begin{definition}
Let $M$ be a matroid.
A \emph{vertical cover} of $M$ is a pair $(F,F')$ of proper flats such that $F\cup F' = E(M)$.
If, in addition, $F$ and $F'$ are modular flats, then $(F,F')$ is a \emph{modular cover}.
A matroid is \emph{round} if it has no vertical cover.
\end{definition}

Thus a matroid is round if and only if there is no partition $(U,U')$ of $E(M)$ such that neither $U$ nor $U'$ is spanning.
Such a partition is said to be a \emph{vertical separation}.
If $X$ is a subset of $E(M)$, then we say that $X$ is round if $M|X$ is round.
If $F$ is a round flat of the matroid $M$ and $F$ is contained in the subset $X\subseteq E(M)$, then clearly $F$ is a round flat of $M|X$.
A round flat is \emph{maximal} if it is not properly contained in a round flat.
For brevity, we refer to a maximal round flat as a \emph{rotunda}.
The set of rotunda of a matroid $M$ is denoted by $\mathcal{R}(M)$.

\begin{proposition}
\label{round-separations}
Let $R$ and $R'$ be distinct rotunda.
Let $(F,F')$ be a vertical cover such that $R\subseteq F$ and $R'\subseteq F'$ and $F\cap F' = R\cap R'$.
Then $R\nsubseteq F'$ and $R'\nsubseteq F$.
\end{proposition}

\begin{proof}
It suffices to prove that $R$ is not contained in $F'$.
Assume this fails.
Then $R$ is contained in $F\cap F' = R\cap R'$, implying that $R$ is a subset of $R'$.
This is impossible since $R$ and $R'$ are distinct rotunda.
\end{proof}

The next result follows from work in \cite{Whittle85}, but we include a proof for completeness.

\begin{proposition}
\label{round-projections}
Let $H$ be a modular hyperplane of the matroid $M$.
Let $X$ be a subset of the cocircuit $E(M)-H$.
Then
\[\{P_{H}(x,y)\colon x,y \in X, r(\{x,y\}) = 2\}\]
is round.
\end{proposition}

\begin{proof}
Let $P$ be the union of all projections onto $H$ of pairs of distinct, non-parallel, elements in $X$.
Thus our aim is to show that $P$ is round.
We assume for a contradiction that $(F,F')$ is a vertical cover of $M|P$, so that $F$ and $F'$ are proper flats of $M|P$ and $F\cup F' = P$.
Note that if $X$ contains fewer than three rank-one flats, then $P$ is either empty or consists of a single rank-one flat.
In this case $P$ is trivially round, so we must assume that $X$ contains at least three rank-one flats.

Let $x$, $y$, and $z$ be distinct rank-one flats in $X$.
Assume that $P_{H}(x,y)$ and $P_{H}(x,z)$ are both in $F$.
We claim that $P_{H}(y,z)$ is also in $F$.
If $z$ is in $\cl(x\cup y)$, then $\cl(x\cup y)=\cl(x\cup z) = \cl(y\cup z)$, and it follows that $P_{H}(x,y)=P_{H}(x,z)=P_{H}(y,z)$, so the claim is true.
Therefore we will assume that $r(x\cup y\cup z) = 3$.
Let $Z$ be $\cl(x\cup y\cup z)$.
Since $H$ is a modular hyperplane and $Z$ is not contained in $H$, it follows that $r(H\cap Z) = 2$.
Now $P_{H}(x,y)$ and $P_{H}(x,z)$ are rank-one flats contained in $H\cap Z$.
If they are not distinct, then $y$ and $z$ are both in the closure of $x\cup P_{H}(x,y)$.
This implies that $z$ is in $\cl(x\cup y)$, contrary to earlier hypothesis.
It follows that $P_{H}(x,y)\cup P_{H}(x,z)$ spans $H\cap Z$, and in particular spans $P_{H}(y,z)$.
Thus $P_{H}(y,z)$ is in $F$, as claimed.
Symmetrically, if $P_{H}(x,y)$ and $P_{H}(x,z)$ are both in $F'$, then so is $P_{H}(y,z)$.

We think of the rank-one flats that have a non-empty intersection with $X$ as the vertices of a complete graph.
If $x$ and $y$ are two such flats, then we colour the edge between $x$ and $y$ red if $P_{H}(x,y)$ is in $F$, and blue if it is in $F'$.
Notice that an edge may be both red and blue.
The previous paragraph shows that if the edges $xy$ and $xz$ are both red (blue), then the edge $yz$ is also red (blue).

Let $x$ be a vertex in this complete graph and assume that every edge incident with $x$ is red.
Then every edge is red, and it follows that $P$ is contained in $F$.
This is impossible since $F$ is a proper flat of $M|P$.
Similarly, it is not possible for every edge incident with $x$ to be blue.

Therefore we can assume that the edge between $x$ and $y$ is red but not blue, and the edge between $x$ and $z$ is blue but not red.
However, if the edge $yz$ is red, then $xz$ is red, and if $yz$ is blue then $xy$ is blue.
In either case we have a contradiction, so the proof is complete.
\end{proof}

\begin{proposition}
\label{mod-hyp-vert-sep}
Let $H$ be a modular hyperplane of the matroid $M$ and let $C^{*}$ be the complementary cocircuit.
Let $(F_{1},F_{2})$ be a vertical cover of $M|H$.
Let $P$ be the union $\cup P_{H}(x,y)$, where $x$ and $y$ range over all distinct rank-one flats contained in $C^{*}$.
Then $P$ is contained in $F_{i}$ for some $i$, and $(F_{i}\cup C^{*},F_{3-i})$ is a vertical cover of $M$.
Moreover, if $(F_{1},F_{2})$ is a modular cover, then so is $(F_{i}\cup C^{*},F_{3-i})$.
\end{proposition}

\begin{proof}
Proposition \ref{round-projections} says that $P$ is a round subset of $H$.
Thus $(F_{1}\cap P, F_{2}\cap P)$ is not a vertical cover of $M|P$, so either $F_{1}\cap P$ or $F_{2}\cap P$ is equal to $P$. 
We assume the former without any loss of generality, so $P\subseteq F_{1}$.
Proposition \ref{mod-hyp-sep-extends} implies that $F_{1}$ is equal to $\cl(F_{1}\cup C^{*})\cap H$.
It follows that $\cl(F_{1}\cup C^{*}) = F_{1}\cup C^{*}$.
Now $F_{1}\cup C^{*}$ is a proper flat of $M$ because $F_{1}$ is a proper flat of $M|H$.
Similarly, $F_{2}$ is a proper flat of $M$.
As $(F_{1}\cup C^{*})\cup F_{2}=E(M)$, it follows that  $(F_{1}\cup C^{*},F_{2})$ is a vertical cover of $M$.

Now we assume that $(F_{1},F_{2})$ is a modular cover of $M|H$.
Then $F_{2}$ is a modular flat of $M|H$ so it immediately follows from \cite{Oxley11}*{Proposition 6.9.7} that $F_{2}$ is also a modular flat in $M$.
It remains only to prove that $F_{1}\cup C^{*}$ is a modular flat of $M$.
To this end, assume that $F$ is a flat of $M$ that is disjoint from $F_{1}\cup C^{*}$.
Thus $F$ is a flat of $M|(F_{2}-F_{1})$.
If we can show that there is no circuit of $M|((F_{1}\cup C^{*})\cup F)$ containing elements from both $F$ and $F_{1}\cup C^{*}$, then the result will follow from Proposition \ref{mod-flat-new-def}.
Assume that $C$ is such a circuit, chosen so that $C\cap C^{*}$ is as small as possible.
Let $f$ be an element of $C\cap F$.
If $C\cap C^{*}=\emptyset$, then Proposition \ref{mod-flat-new-def} implies that $F_{1}$ is not a modular flat of $M|H$, which is a contradiction.
Therefore $C\cap C^{*}\ne \emptyset$.
If $C\cap C^{*}$ contains a single element, $x$, then $C$ certifies that $x$ is in $\cl(H) = H$, a contradiction.
Therefore we let $x$ and $y$ be distinct elements in $C\cap C^{*}$.
Let $p$ be in $P_{H}(x,y)$.
Thus $\{x,y,p\}$ is a circuit and $p$ is in $P$, and hence in $F_{1}$.
We perform strong circuit elimination on $C$ and $\{x,y,p\}$ to obtain $C'\subseteq (C-x)\cup\{y,p\}$, a circuit that contains $f$.
It must contain $p$, since otherwise it is properly contained in $C$.
But now $C'$ is contained in $F_{1}\cup C^{*}\cup F$, and it contains elements from both $F$ and $F_{1}\cup C^{*}$.
Since $C'\cap C^{*}$ is strictly smaller than $C\cap C^{*}$, we have contradicted our choice of $C$, so the proof is complete.
\end{proof}

The following result provides a partial converse to Proposition \ref{mod-hyp-vert-sep}.

\begin{proposition}
\label{prop3}
Let $H$ be a modular hyperplane of the matroid $M$ and let $C^{*}$ be the complementary cocircuit.
Let $(F,F')$ be a modular cover of $M$ such that $F'$ is contained in $H$.
If $F'\ne H$, then $(F\cap H,F'\cap H)$ is a modular cover of $M|H$.
\end{proposition}

\begin{proof}
Note that $C^{*}$ is contained in $F$ because $F'$ contains no element of $C^{*}$.
Let $P$ be the union of $P_{H}(x,y)$ where $x$ and $y$ range over distinct rank-one flats in $C^{*}$.
Since $F$ contains $C^{*}$, and $P$ is spanned by $C^{*}$ it follows that $P$ is a subset of $F$.
Now $F\cap H$ is the intersection of two modular flats, so Proposition \ref{modular-intersection} implies that it is a modular flat of $M$, and hence of $M|H$.
Because $F'$ is contained in $H$ it is also true that $F'\cap H = F'$ is a modular flat of $M|H$.
By hypothesis $F'\cap H$ is a proper flat of $M|H$.
Furthermore, $F\cap H$ is a proper flat of $M|H$, or else $F$ contains $H\cup C^{*}=E(M)$, contradicting the fact that $F$ is a proper flat of $M$.
Therefore $(F\cap H, F'\cap H)$ is a modular cover of $M|H$.
\end{proof}

\begin{proposition}
\label{prop2}
Let $H$ be a modular hyperplane of the matroid $M$ and let $C^{*}$ be the complementary cocircuit.
If $F$ is a round flat not contained in $H$, then $F\subseteq \cl(C^{*})$.
\end{proposition}

\begin{proof}
Assume this fails.
Then $F\cap C^{*}$ does not span $F$.
It is also true that $F\cap H$ does not span $F$, as $\cl(F\cap H)\subseteq \cl(H) = H$ and $F$ is not contained in $H$.
Therefore $(F\cap H, F\cap C^{*})$ is a vertical cover of $M|F$, and this contradicts the fact that $M|F$ is round.
\end{proof}

\begin{proposition}
\label{cocircuit-closure-rotunda}
Let $H$ be a modular hyperplane of the matroid $M$ and let $C^{*}$ be the complementary cocircuit.
Then $\cl(C^{*})$ is a rotunda.
Furthermore, every other rotunda of $M$ is contained in $H$.
\end{proposition}

\begin{proof}
Let $R$ be $\cl(C^{*})$.
Assume that $R$ is not round, and let $(F,F')$ be a vertical cover of $M|R$.
Let $P$ be the union $\cup P_{H}(x,y)$, where $x$ and $y$ range over all distinct rank-one flats contained in $C^{*}$.
Note that $P$ is contained in $R\cap H$.
Proposition \ref{round-projections} says that $P$ is round.
It follows that one of $F\cap P$ or $F'\cap P$ is equal to $P$.
Without loss of generality we will assume the former.

If $F'$ contains $C^{*}$, then it contains $R$, which is impossible as $(F,F')$ is a vertical cover of $R$.
Therefore we choose $x\in C^{*}-F'$.
The same argument shows we can choose $y\in C^{*}-F$.
Note that $x$ and $y$ are not parallel, since $x$ is in $F-F'$ and $y$ is in $F'-F$.
Let $p$ be in $P_{H}(x,y)$, so that $p$ is in $P$, and hence in $F$.
As $\{x,y,p\}$ is a circuit and both $x$ and $p$ belong to the flat $F$ it follows that $y$ is in $F$, contrary to assumption.
Therefore $R$ is round.

Let $Z$ be any flat that properly contains $R$.
Note that $Z\cap H$ is a flat that does not contain any element of $C^{*}$.
Therefore $(Z\cap H,R)$ is a vertical cover of $Z$, a contradiction.
This shows that $R$ is a maximal round flat, which is to say, a rotunda.

Finally, let $Z$ be a rotunda that is not contained in $H$.
By Proposition \ref{prop2}, we see that $Z$ is contained in $R$.
As $Z$ and $R$ are both rotunda it now follows that $Z=R$.
\end{proof}

\begin{proposition}
\label{mod-hyp-round-intersect}
Let $H$ be a modular hyperplane of the matroid $M$.
Let $C^{*}$ be the complementary cocircuit.
Then $\cl(C^{*})\cap H$ is round.
\end{proposition}

\begin{proof}
Let $R$ be $\cl(C^{*})$.
Assume for a contradiction that $(F,F')$ is a vertical cover of $R\cap H$.
Let $P$ be the union $\cup P_{H}(x,y)$ where $x$ and $y$ range over all distinct rank-one flats contained in $C^{*}$.
Note that $P$ is contained in $R\cap H$.
Proposition \ref{round-projections} says that $P$ is round.
Therefore $(F\cap P,F'\cap P)$ is not a vertical cover of $P$, so 
we can assume without loss of generality that $P$ is contained in $F$.
Applying Proposition \ref{mod-hyp-sep-extends}, we see that  $\cl(F\cup C^{*})\cap H$ is equal to $F$.
Thus $\cl(C^{*})\cap H = R\cap H$ is contained in $F$.
This contradicts the fact that $(F,F')$ is a vertical cover of $R\cap H$.
\end{proof}

\section{Supersolvability and saturation}

The following definition was introduced by Stanley \cite{stanley}.

\begin{definition}
The rank\dash $r$ matroid $M$ is \emph{supersolvable} if it has a chain of modular flats $F_{0}\subseteq F_{1}\subseteq \cdots\subseteq F_{r}$, where $r(F_{i}) = i$ for each $i$.
\end{definition}

We can give an equivalent, recursive, definition: if $r(M)>0$ then $M$ is supersolvable if it contains a modular hyperplane $H$ such that $M|H$ is supersolvable.
Note that every rank-zero matroid is trivially supersolvable.

\begin{definition}
A matroid is \emph{saturated} if every round flat is modular.
\end{definition}

\begin{proposition}
\label{saturated-restriction}
Let $F$ be a flat of the saturated matroid $M$.
Then $M|F$ is saturated.
\end{proposition}

\begin{proof}
Let $R$ be a round flat of $M|F$.
Then $R$ is a round flat of $M$ so it is modular in $M$.
Now \cite{Oxley11}*{Proposition 6.9.5} implies that $R$ is a modular flat of $M|F$.
\end{proof}

If $M$ is supersolvable and saturated and $H$ is a modular hyperplane such that $M|H$ is supersolvable, then it follows from Proposition \ref{saturated-restriction} that $M|H$ is supersolvable and saturated.

\begin{proposition}
\label{supsol-mod-sep}
Let $M$ be a saturated matroid.
Let $H$ be a modular hyperplane of $M$ and let $C^{*}$ be the complementary cocircuit.
If $C^{*}$ is non-spanning, then $(H,\cl(C^{*}))$ is a modular cover of $M$.
\end{proposition}

\begin{proof}
Certainly $H$ is a proper flat of $M$, and $C^{*}$ is non-spanning by hypothesis.
Therefore $(H,\cl(C^{*}))$ is a vertical cover.
We have assumed that $H$ is a modular flat.
Proposition \ref{cocircuit-closure-rotunda} says that $\cl(C^{*})$ is round.
Since $M$ is saturated, it follows that $\cl(C^{*})$ is modular, so the proof is complete.
\end{proof}

\begin{proposition}
\label{supersolvable-components}
Let $M$ be a matroid.
Then $M$ is supersolvable if and only if each of its connected components is supersolvable.
Similarly $M$ is saturated if and only if each of its connected components is saturated.
\end{proposition}

\begin{proof}
This result will follow by an easy inductive argument if we can prove it in the case when $M$ has exactly two connected components.
Therefore we will assume that $M=M_{1}\oplus M_{2}$, where $M_{1}$ and $M_{2}$ are non-empty connected matroids.
For $i=1,2$, let $r_{i}$ be $r(M_{i})$.

Assume that $M_{1}$ and $M_{2}$ are supersolvable.
For $i=1,2$, let $F_{0}^{i}\subseteq F_{1}^{i}\subseteq \cdots\subseteq F_{r_{i}}^{i}$ be a chain of modular flats in $M_{i}$ such that each $F_{j}^{i}$ has rank $j$.
Using \cite{Oxley11}*{Corollary 6.9.10} we see that each $F_{j}^{1}\cup F_{k}^{2}$ is a modular flat of $M$.
Now it is easy to confirm that the chain
\[
F_{0}^{1}\subseteq F_{1}^{1}\subseteq \cdots\subseteq F_{r_{1}}^{1}\subseteq F_{r_{1}}^{1}\cup F_{1}^{2}\subseteq F_{r_{1}}^{1}\cup F_{2}^{2}\cdots\subseteq F_{r_{1}}^{1}\cup F_{r_{2}}^{2}
\]
certifies that $M$ is supersolvable.

For the other direction, assume that $M$ is supersolvable.
Assume for a contradiction that either $M_{1}$ or $M_{2}$ is not supersolvable.
We will assume that amongst such counterexamples, $M$ is as small as possible.
Now $M$ has a modular hyperplane $H$ such that $M|H$ is supersolvable.
The complement of $H$ is a cocircuit, and is therefore contained in either $M_{1}$ or $M_{2}$.
Without loss of generality we assume that $H$ contains $E(M_{2})$.
Now $M|H = (M_{1}|H)\oplus M_{2}$.
The minimality of $M$ means that $M_{1}|H$ and $M_{2}$ are both supersolvable.
But \cite{Oxley11}*{Corollary 6.9.10} implies that $H\cap E(M_{1})$ is a modular flat of $M_{1}$.
It is the complement of a cocircuit of $M_{1}$, so $H\cap M_{1}$ is a modular hyperplane of $M_{1}$, and restricting to this hyperplane produces a supersolvable matroid.
This shows that $M_{1}$ too is supersolvable, so the proof of this direction is complete.

From \cite{Oxley11}*{Corollary 6.9.10} we see that $E(M_{1})$ and $E(M_{2})$ are modular flats of $M$.
It follows from \cite{Oxley11}*{Proposition 6.9.5} that a flat of $M_{i}$ is modular in $M$ if and only if it is modular in $M_{i}$.
If $F$ is a round flat of $M$ then $F\subseteq E(M_{1})$ or $F\subseteq E(M_{2})$ because otherwise $(F\cap E(M_{1}),F\cap E(M_{2}))$ is a vertical cover of $M|F$.
In fact, the round flats of $M$ are exactly the round flats of $M_{1}$ along with the round flats of $M_{2}$.
From these considerations we can easily see that $M$ is saturated if and only if $M_{1}$ and $M_{2}$ are saturated.
\end{proof}

\subsection{Chordality for matroids}
\label{CHORDALITY}

We shall start this section by justifying the Venn diagram in Figure \ref{Fig4}.
Recall that if $C$ is a matroid circuit, then a chord of $C$ is an element $x\notin C$ such that $A\cup z$ and $B\cup z$ are both circuits for some partition of $C$ into sets $A$ and $B$.
A matroid is $C$\dash chordal if every circuit with at least four elements has a chord.

As we discussed in the introduction, the matroid in Figure \ref{Fig2} is supersolvable but not saturated.
To see that it is not $C$\dash chordal, note that $\{3,5,6,7\}$ has no chord.
Because the only round flats of $U_{3,6}$ are the empty set, the singleton sets, and the entire ground set, we can easily confirm that every round flat is modular, so $U_{3,6}$ is saturated.
It has no modular hyperplane, so it is not supersolvable, and no circuit has a chord so it is not $C$\dash chordal.
Recall that $\mathcal{W}^{4}$ is the rank-three matroid with ground set $\{a,b,c,d,e,f\}$ and non-spanning circuits $\{a,b,d\}$, $\{b,c,e\}$, and $\{a,c,f\}$.
It is easy to confirm that every circuit of size four has a chord.
However no line is modular, so $\mathcal{W}^{4}$ is not supersolvable, and it also follows that it is not saturated.

We will leave as an exercise the fact that the Fano matroid $F_{7}$ is supersolvable, saturated, and $C$\dash chordal.
Cordovil et al.\ note that $M^{*}(K_{3,3})$ is not supersolvable \cite{CFK04}.
It is an easy exercise to see that it is saturated and $C$\dash chordal.
Finally, let $M$ be the rank-three matroid with ground set $\{p,a,b,c,d,e,f,x\}$ where the non-spanning circuits are $\{p,a,b,c\}$, $\{p,d,e,f\}$, $\{a,d,x\}$, $\{b,e,x\}$, and $\{c,f,x\}$.
Now $\{p,a,b,c\}$ and $\{p,d,e,f\}$ are both modular hyperplanes, and we can easily confirm that $M$ is supersolvable.
On the other hand, $\{a,d,x\}$ is a round hyperplane that has empty intersection with the rank-two flat $\{b,f\}$.
Hence $\{a,d,x\}$ is not modular and therefore $M$ is not saturated.
On the other hand, a simple case-analysis shows that $M$ is $C$\dash chordal.
We can finish the justification of Figure \ref{Fig4} by proving that every supersolvable saturated matroid is $C$\dash chordal.
In fact, we prove something slightly stronger.

\begin{proposition}
\label{SSS-implies-chordal}
Let $C$ be a circuit in the supersolvable saturated matroid $M$ and assume that $|C|\geq 4$.
There exist distinct elements $x,y\in C$ and an element $z\notin C$ such that $\{x,y,z\}$ and $(C-\{x,y\})\cup z$ are circuits of $M$.
\end{proposition}

\begin{proof}
Let $M$ be a smallest possible counterexample to the result.
If $r(M)\leq 2$ then the result holds vacuously, so $r(M)\geq 3$.
Let $H$ be a modular hyperplane of $M$ such that $M|H$ is supersolvable and saturated.
Let $C^{*}$ be the complement of $H$.

Choose $C$ to be an arbitrary circuit of $M$ such that $|C|\geq 4$.
If $C$ is a circuit of $M|H$, then the result holds by induction.
Therefore $C\cap C^{*}$ is non-empty.
Because $H$ is a flat it follows that $C\cap C^{*}$ contains distinct elements $x$ and $y$.
Let $L$ be $\cl(\{x,y\})$.
Note that $L$ contains an element in $P_{H}(x,y)$, so that $L$ is a rank-two flat containing at least three rank-one flats.
Now it is easy to confirm that $L$ is a round flat.
Since $M$ is saturated, it follows that $L$ is modular.
Note that $C$ contains exactly two elements of $L$ because $|C|\geq 4$.
Now
\[
r(\cl(C-L)\cap L) = r(C-L) +r(L) - r(C) = (|C|-2)+2-(|C|-1)=1.
\]
Therefore we choose an element $z$ which is in $\cl(C-L)\cap L$.
Note that neither $x$ nor $y$ is in $\cl(C-L)$, or else $C$ properly contains a circuit.
Therefore $z$ is in $L-\{x,y\}$ and $\{x,y,z\}$ is a circuit.
Let $C'\subseteq (C-L)\cup z$ be a circuit that contains $z$.
Now $(C'\cup\{x,y\})-z$ contains a circuit, by circuit elimination with $C'$ and $\{x,y,z\}$.
But $(C'\cup\{x,y\})-z$ is a subset of $C$, so $(C'\cup\{x,y\})-z=C$.
It follows that $C' = (C-L)\cup z$.
Thus $(C-L)\cup z$ and $\{x,y,z\}$ are both circuits and $M$ is not a counterexample after all.
\end{proof}

In the next results we justify using supersolvable saturated matroids as analogues for chordal graphs.

\begin{proposition}
\label{prop1}
Let $G$ be a graph, and let $F$ be a flat of $M(G)$.
Then $F$ is round if and only if $G[F]$ is a clique.
\end{proposition}

\begin{proof}
Let $M$ be $M(G)$.
Assume that $G[F]$ is not a clique.
Let $u$ and $v$ be distinct vertices in $G[F]$ that are not adjacent.
Let $U$ be the set of edges in $F$ that are incident with $u$, and let $U'$ be $F-U$.
If $f\in F$ is an edge incident with $u$, there is no cycle contained in $U'\cup f$ that contains $f$.
This shows that $\cl(U')$ is a proper flat of $M|F$.
The same argument shows that $\cl(U)$ is a proper flat of $M|F$, so $(\cl(U),\cl(U'))$ is a vertical cover of $M|F$.
Thus $F$ is not round.

For the other direction, assume that $G[F]$ is a clique, but that $(U,U')$ is a vertical cover of $G[F]$.
We colour the edges of $F$ red if they are in $U$, and blue if they are in $V$.
Note that an edge may be both red and blue.
Let $v$ be an arbitrary vertex of $G[F]$.
The set of edges incident with $v$ spans $F$, since $G[F]$ is a clique.
If all the edges of $F$ incident with $v$ are red, then $U$ contains $F$, a contradiction.
By symmetry, we can now let $e,f\in F$ be edges incident with $v$ so that $e$ is red but not blue, and $f$ is blue but not red.
Let $g$ be the edge of $F$ so that $\{e,f,g\}$ is the edge-set of a triangle.
If $g$ is red, then $f$ is also red, and if $g$ is blue, then $e$ is blue, and in either case we have a contradiction.
\end{proof}

\begin{corollary}
\label{graphic-implies-saturation}
Let $G$ be a graph.
Then $M(G)$ is a saturated matroid.
\end{corollary}

\begin{proof}
From Proposition \ref{prop1} we see that every round flat of $M(G)$ is a clique of $G$, and any such flat is modular by \cite{Oxley11}*{Proposition 6.9.11}.
The result follows.
\end{proof}

The next result is a consequence of \cite{stanley}*{Proposition 2.8}.

\begin{proposition}
\label{chordal-iff-supersolv}
Let $G$ be a graph.
Then $G$ is chordal if and only if $M(G)$ is supersolvable.
\end{proposition}

The next result implies the known fact \cite{Golumbic04}*{Proposition 4.16} that in a chordal graph the number of maximal cliques does not exceed the number of vertices.

\begin{proposition}
Let $M$ be a supersolvable matroid.
Then $M$ has at most $r(M)$ rotunda.
\end{proposition}

\begin{proof}
Let $H$ be a modular hyperplane of $M$ such that $M|H$ is supersolvable.
Any rotunda of $M$ that is contained in $H$ is a rotunda of $M|H$.
But $M|H$ has at most $r(M)-1$ rotunda by induction, and Proposition \ref{cocircuit-closure-rotunda} says there is exactly one rotunda of $M$ that is not a rotunda of $M|H$.
The result follows.
\end{proof}

\section{Reduced clique graphs and rotunda graphs}
\label{RC-and-ROTgraphs}

Let $G$ be a chordal graph.
The \emph{clique graph} of $G$, denoted $C(G)$, has the maximal cliques of $G$ as its vertices.
Two distinct maximal cliques are adjacent in $C(G)$ if and only if they have at least one vertex in common.
Our focus will be the \emph{reduced clique graph}, $C_{R}(G)$, which was introduced in \cite{GHP95}.
The vertices of $C_{R}(G)$ are again the maximal cliques of $G$.
Let $C_{1}$ and $C_{2}$ be distinct maximal cliques of $G$.
We say that $C_{1}$ and $C_{2}$ are a \emph{separating pair} if there is at least one vertex in $C_{1}\cap C_{2}$ and any path from a vertex of $C_{1}-C_{2}$ to a vertex of $C_{2}-C_{1}$ uses a vertex in $C_{1}\cap C_{2}$.
Now $C_{R}(G)$ is the subgraph of $C(G)$ where two maximal cliques are adjacent if and only if they form a separating pair.
We now define a matroid analogue of this graph.

\begin{definition}
Let $M$ be a supersolvable saturated matroid.
Recall that $\mathcal{R}(M)$ is the family of rotunda of $M$.
The \emph{rotunda graph} $R(M)$ is the graph with $\mathcal{R}(M)$ as its vertex set.
The rotunda $R_{1}$ and $R_{2}$ are adjacent in $R(M)$ if $R_{1}\cap R_{2}\ne \emptyset$ and there is a modular cover $(F_{1},F_{2})$ such that $R_{i}\subseteq F_{i}$ for $i=1,2$, and $F_{1}\cap F_{2} = R_{1}\cap R_{2}$.
In this case we say that the modular cover $(F_{1},F_{2})$ \emph{certifies} the adjacency of $R_{1}$ and $R_{2}$.
\end{definition}

The next result allows us to prove statements about rotunda graphs inductively.

\begin{proposition}
\label{mod-hyp-rot-graph}
Let $M$ be a supersolvable saturated matroid and let $H$ be a modular hyperplane of $M$ such that $M|H$ is supersolvable.
Let $C^{*}$ be the complement of $H$ and let $R$ be $\cl(C^{*})$.
Then $R$ is a rotunda of $M$ and either:
\begin{enumerate}[label = \textup{(\alph*)}]
\item $R\cap H$ is a rotunda of $M|H$ and \[\mathcal{R}(M)=(\mathcal{R}(M|H)-\{R\cap H\})\cup\{R\},\qquad\text{or}\]
\item $R\cap H$ is properly contained in a rotunda of $M|H$ and \[\mathcal{R}(M)=\mathcal{R}(M|H)\cup\{R\}.\]
\end{enumerate}
If case \textup{(a)} holds then $R(M|H)$ is obtained from $R(M)$ by relabelling $R$ as $R\cap H$.
If case \textup{(b)} holds then $R(M|H)$ is obtained from $R(M)$ by deleting $R$.
\end{proposition}

\begin{proof}
Note that $M|H$ is saturated as well as supersolvable (Proposition \ref{saturated-restriction}).
Proposition \ref{cocircuit-closure-rotunda} says that $R$ is a rotunda of $M$, and that moreover it is the unique rotunda of $M$ that is not contained in $H$.
Now it is an easy exercise to prove that every other rotunda of $M$ is a rotunda of $M|H$.
This shows $\mathcal{R}(M)\subseteq \mathcal{R}(M|H)\cup\{R\}$.

Proposition \ref{mod-hyp-round-intersect} says that $R\cap H$ is a round flat of $M|H$.
First assume that $R\cap H$ is a maximal round flat of $M|H$.
Then $R\cap H$ is a rotunda of $M|H$ but not of $M$, since $R\cap H$ is properly contained in $R$.
So in this case $\mathcal{R}(M)$ is contained in $(\mathcal{R}(M|H)-\{R\cap H\})\cup\{R\}$.
Now let $Z$ be a rotunda of $M|H$ that is not equal to $R\cap H$.
We will prove that $Z$ is a rotunda of $M$.
Assume otherwise.
Because $Z$ is a round flat of $M|H$, and hence of $M$, it is properly contained in a rotunda of $M$.
Let this rotunda be $Z'$.
Now $Z'$ is not contained in $H$, because in this case $Z$ and $Z'$ would both be rotunda of $M|H$, and then $Z$ cannot be properly contained in $Z'$.
So $Z'$ is a rotunda of $M$ that is not contained in $H$, and hence $Z'=R$.
Thus $Z$ is contained in $R\cap H$.
Because $Z$ is not properly contained in a round flat of $M|H$ we deduce that $Z = R\cap H$, contrary to hypothesis.
Thus $Z$ is a rotunda of $M$ and we have shown that when $R\cap H$ is a rotunda of $M|H$, the set $\mathcal{R}(M)$ is equal to $(\mathcal{R}(M|H)-\{R\cap H\})\cup\{R\}$ and case (a) holds.

Next we assume that $R\cap H$ is not a rotunda of $M|H$.
We have already shown that $\mathcal{R}(M)$ is contained in $\mathcal{R}(M|H)\cup\{R\}$.
Let $Z$ be a rotunda of $M|H$ and assume that $Z$ is not a rotunda of $M$.
Then $Z$ is properly contained in $Z'$, a rotunda of $M$.
As in the previous paragraph, $Z'=R$, so $Z$ is contained in $R\cap H$.
Again we deduce that $Z=R\cap H$, and we have a contradiction to $Z$ being a rotunda of $M|H$.
So in the case $\mathcal{R}(M)$ is equal to $\mathcal{R}(M|H)\cup\{R\}$.
Furthermore, $R\cap H$ is a round flat of $M|H$ but not a rotunda, so it must be properly contained in a rotunda of $M|H$.
Thus case (b) holds.

Assume case (a) holds.
We let $Z_{1}$ and $Z_{2}$ be distinct rotunda of $M$, where $Z_{1}$ is not equal to $R$.
Thus $Z_{1}$ is a rotunda of $M|H$.
Either $Z_{2}$ is equal to $R$ or it is not.
In the former case $Z_{2}\cap H = R\cap H$ and in the
latter $Z_{2}\cap H = Z_{2}$.
In either case $Z_{2}\cap H$ is a rotunda of $M|H$.
We will prove that $Z_{1}$ and $Z_{2}\cap H$ are adjacent in $R(M|H)$ if and only if $Z_{1}$ and $Z_{2}$ are adjacent in $R(M)$, and this will show that $R(M|H)$ is obtained from $R(M)$ by relabelling $R$ as $R\cap H$.

Assume that $(F_{1},F_{2})$ is a modular cover of $M$ that certifies the adjacency of $Z_{1}$ and $Z_{2}$ in $R(M)$.
Thus $F_{1}$ and $F_{2}$ are proper modular flats of $M$ and $F_{1}\cup F_{2}=E(M)$.
Moreover $F_{1}\cap F_{2}=Z_{1}\cap Z_{2}$.
Assume that either $F_{1}$ or $F_{2}$ contains $H$.
Since Proposition \ref{round-separations} implies that neither $F_{1}$ nor $F_{2}$ contains $Z_{1}\cup Z_{2}$, we deduce that $Z_{2}=R$ and $F_{1}=H$.
Now
\[R\cap H \subseteq F_{1}\cap F_{2} = Z_{1}\cap Z_{2}\]
so $Z_{1}$ contains $R\cap H$.
Since $Z_{1}$ and $R\cap H$ are both rotunda of $M|H$, we see that $Z_{1}=R\cap H$, and in this case $Z_{1}$ is properly contained in $Z_{2}$.
This is impossible, so $F_{1}\cap H$ or $F_{2}\cap H$ are proper flats of $M|H$.
Moreover, their union is equal to $H$.

Since $F_{1}$ and $F_{2}$ are modular flats of $M$ it follows that $F_{1}\cap H$ and $F_{2}\cap H$ are modular flats of $M$ (Proposition \ref{modular-intersection}), and hence modular flats of $M|H$.
Furthermore,
\[
(F_{1}\cap H)\cap(F_{2}\cap H)
=(F_{1}\cap F_{2})\cap H
=(Z_{1}\cap Z_{2})\cap H
=Z_{1}\cap (Z_{2}\cap H).
\]
Now we see that $(F_{1}\cap H,F_{2}\cap H)$ is a modular cover of $M|H$, and that it certifies the adjacency of $Z_{1}$ and $Z_{2}\cap H$ in $R(M|H)$.

For the other direction, assume $Z_{1}$ and $Z_{2}\cap H$ are adjacent in $R(M|H)$, and let $(F_{1},F_{2})$ be a modular cover of $M|H$ that certifies their adjacency.
Let $P$ be the union $\cup P_{H}(x,y)$, where $x$ and $y$ range over distinct rank-one flats in $C^{*}$.
We apply Proposition \ref{mod-hyp-vert-sep} and see that $P$ is contained in either $F_{1}$ or $F_{2}$.

Assume that $Z_{2}=R$.
Then $P$ is contained in $Z_{2}\cap H\subseteq F_{2}$.
In this case Proposition \ref{mod-hyp-vert-sep} implies that $(F_{1},F_{2}\cup C^{*})$ is a modular cover of $M$.
Moreover, \[F_{1}\cap (F_{2}\cup C^{*}) = F_{1}\cap F_{2} = Z_{1}\cap (Z_{2}\cap H) = Z_{1}\cap Z_{2}.\]
Thus $(F_{1},F_{2}\cup C^{*})$ certifies the adjacency of $Z_{1}$ and $Z_{2}$ in $R(M)$.
Next we assume that $Z_{2}\ne R$, so that $Z_{1}$ and $Z_{2}$ are both rotunda of $M|H$.
We again apply Proposition \ref{mod-hyp-vert-sep} and see that $(F_{i}\cup C^{*},F_{3-i})$ is a modular cover of $M$ for some $i\in\{1,2\}$, and as before we can see that $(F_{i}\cup C^{*},F_{3-i})$ certifies the adjacency of $Z_{1}$ and $Z_{2}$ in $R(M)$.
Thus we are now finished with case (a).

Assume case (b) holds.
Let $Z_{1}$ and $Z_{2}$ be two rotunda of $M|H$.
We can use exactly the same arguments as in the previous paragraphs to show that $Z_{1}$ and $Z_{2}$ are adjacent in $R(M|H)$ if and only if they are adjacent in $R(M)$.
Thus $R(M|H)$ is obtained from $R(M)$ by deleting the rotunda $R$ and the proof is complete.
\end{proof}

\subsection{Rotunda graphs vs.\ reduced clique graphs}
\label{VS}

In this section we compare rotunda graphs and reduced clique graphs.
Ultimately we will show that they are identical classes of graphs.
We also consider the connection between the reduced clique graph of $G$ and the rotunda graph of $M(G)$ when $G$ is a chordal graph.

\begin{proposition}
\label{edges-in-ROTgraph}
Let $G$ be a chordal graph.
Then the maximal cliques of $G$ are the rotunda of $M(G)$, and every edge in $R(M(G))$ is an edge in $C_{R}(G)$.
\end{proposition}

\begin{proof}
The first statement follows from Proposition \ref{prop1}.
Let $M$ stand for $M(G)$, so that we identify the vertices of $C_{R}(G)$ and the vertices of $R(M)$.
Let $R_{1}$ and $R_{2}$ be rotunda that are adjacent in $R(M)$, and let $C_{1}$ and $C_{2}$ be the corresponding maximal cliques of $G$.
We will show that $C_{1}$ and $C_{2}$ are adjacent in $C_{R}(G)$.
Let $(F_{1},F_{2})$ be a modular cover of $M$ certifying the adjacency of $R_{1}$ and $R_{2}$, so that $R_{i}\subseteq F_{i}$ for $i=1,2$, and $F_{1}\cap F_{2} = R_{1}\cap R_{2}$.

Because $R_{1}$ and $R_{2}$ are adjacent in $R(M)$, they have a non-empty intersection, which means that $C_{1}$ and $C_{2}$ share at least two vertices.
Let $S$ be the set of vertices in both $C_{1}$ and $C_{2}$.
Thus $|S|\geq 2$.
If $C_{1}$ and $C_{2}$ form a separating pair, then there is nothing left for us to prove.
Therefore we will let $P$ be an $S$\dash avoiding path from $a_{1}\in C_{1}-C_{2}$ to $a_{2}\in C_{2}-C_{1}$.

Let $u$ be an arbitrary vertex in $S$.
Assume that every edge of $C_{1}$ incident with $u$ is in $F_{2}$.
Then $R_{1}$ is contained in $F_{2}$, which contradicts Proposition \ref{round-separations}.
Therefore we let $e_{1}$ be an edge of $C_{1}$ that is incident with $u$ and not in $F_{2}$.
By the same reasoning, we can let $e_{2}$ be an edge of $C_{2}$ that is incident with $u$ and not in $F_{1}$.
Assume that $e_{i}$ joins $u$ to $b_{i}$ for $i=1,2$.
Note that $b_{1}$ is in $C_{1}-C_{2}$, or else $e_{1}$ would be in $R_{2}\subseteq F_{2}$.
Similarly $b_{2}$ is in $C_{2}-C_{1}$.

We obtain the cycle $D$ from $P$ by appending the edges $e_{1}$ and $e_{2}$ as well as $a_{1}b_{1}$ and $a_{2}b_{2}$.
(This assumes that $a_{1}\ne b_{1}$; if $a_{1}=b_{1}$ then we do not append $a_{1}b_{1}$.
The same comment applies if $a_{2}=b_{2}$.)

Note that $D$ is not contained in $F_{2}$, as $e_{1}$ is not in $F_{2}$.
Because $F_{2}$ is a modular flat we can apply Proposition \ref{short-circuit} and deduce that there is an element $x\in F_{2}\cap \cl(D-F_{2})$.
Thus $D'\cup x$ is a cycle of $G$ for some subset $D'\subseteq D-F_{2}$.
Because $(D-F_{2})\cup x$ is a circuit of $M$ it follows that $x$ is in $F_{1}$ as well as $F_{2}$.
Therefore $x$ is in $R_{1}\cap R_{2}$, so $x$ joins two vertices of $S$.
Let $v$ be a vertex incident with $x$ such that $v$ is not $u$.
Thus $v$ is in the cycle $D$, so $v$ is either an internal vertex of $P$, or is equal to one of $a_{1}$, $b_{1}$, $a_{2}$, or $b_{2}$.
But none of the internal vertices of $P$ is in $S$, and $a_{1},b_{1}$ are in $C_{1}-C_{2}$ while $a_{2},b_{2}$ are in $C_{2}-C_{1}$.
Therefore we have a contradiction that completes the proof.
\end{proof}

From the previous result we know that $R(M(G))$ is a subgraph of $C_{R}(G)$.
To see that $R(M(G))$ and $C_{R}(G)$ need not be equal, we let $G$ be the path with two edges.
Thus $G$ is a tree and is therefore chordal.
There are two maximal cliques in $G$, and $M(G)$ has two rotunda.
However $C_{R}(G)$ consists of two vertices joined by an edge, whereas $R(M(G))$ consists of two isolated vertices, since the two rotunda of $M(G)$ are disjoint.
The next result shows that sufficient connectivity prevents this situation from happening.

\begin{proposition}
\label{2-con-same}
Let $G$ be a chordal graph that is $2$-connected.
Then $C_{R}(G) = R(M(G))$.
\end{proposition}

\begin{proof}
We identify the vertices of $C_{R}(G)$ and $R(M)$.
By virtue of Proposition \ref{edges-in-ROTgraph}, it suffices to show that every edge of $C_{R}(G)$ is also an edge of $R(M)$.
To this end let $C_{1}$ and $C_{2}$ be maximal cliques of $G$ that are adjacent in $C_{R}(G)$.
Let $R_{i}$ be the edge set of $C_{i}$ for $i=1,2$.
Then $R_{1}$ and $R_{2}$ are rotunda of $M$.
We will show they are adjacent in $R(M)$.

Set $S$ to be the set of vertices in both $C_{1}$ and $C_{2}$.
Since $C_{1}$ and $C_{2}$ are adjacent in $C_{R}(G)$ it follows that $S\ne \emptyset$.
For each $i=1,2$, let $a_{i}$ be a vertex in $C_{i}-C_{3-i}$.

\begin{claim}
$R_{1}\cap R_{2}\ne \emptyset$
\end{claim}

\begin{proof}
This claim holds if $|S|\geq 2$, because then any edge joining two vertices of $S$ is in $R_{1}\cap R_{2}$.
So assume that $|S|=1$ and let $v$ be the unique vertex of $S$.
Now $C_{1}$ and $C_{2}$ form a separating pair, so $a_{1}$ and $a_{2}$ are in different connected components of $G-S=G-v$, but this contradicts the fact that $G$ is $2$\dash connected.
\end{proof}

Now we know that $R_{1}$ and $R_{2}$ are not disjoint we can complete the proof by constructing a modular cover to certify their adjacency in $R(M)$.
Let $U_{1}$ be the set of edges that are contained in $S$\dash avoiding paths having $a_{1}$ as a terminal vertex.
Observe that every edge incident with $a_{1}$ is in $U_{1}$.
Let $U_{2}$ be the set of edges of $G$ not in $U_{1}$.
Thus $(U_{1},U_{2})$ is a partition of the edge set.

\begin{claim}
$(U_{1},U_{2})$ is a vertical separation of $M$.
\end{claim}

\begin{proof}
We must prove that neither $U_{1}$ nor $U_{2}$ is spanning in $M$.
Let $e$ be any edge incident with $a_{2}$.
We claim that $e$ is not in $U_{1}$.
Assume otherwise, and let $P$ be an $S$\dash avoiding path with $a_{1}$ as a terminal vertex, where $P$ contains $e$.
Since $e$ is incident with $a_{2}$, we can let $P'$ be a subpath of $P$ from $a_{1}$ to $a_{2}$.
As $C_{1}$ and $C_{2}$ form a separating pair, it follows that $P'$ contains a vertex of $S$.
But the end vertices of $P'$ are $a_{1}$ and $a_{2}$, and neither is in $S$, so $P'$ has an internal vertex in $C_{1}\cap C_{2}$.
Thus $P$ does as well, a contradiction.
Therefore $e$ is not in $U_{1}$.

Assume that $U_{1}$ is spanning.
Let $e$ be an edge incident with $a_{2}$.
Then $e$ is in $U_{2}$ by the previous paragraph.
Since it is in the closure of $U_{1}$, we can let $D$ be a cycle containing $e$ such that every other edge of $D$ is in $U_{1}$.
In particular, this means that $a_{2}$ is incident with an edge of $U_{1}$, contrary to the previous paragraph.
So $U_{1}$ is not spanning.

Similarly, if $U_{2}$ is spanning, then we let $e$ be an edge incident with $a_{1}$.
Then $e$ is not in $U_{2}$, so we can let $D$ be a cycle that contains $e$, where all the other edges of $D$ are in $U_{2}$.
This implies that an edge incident with $a_{1}$ is in $U_{2}$, which contradicts an earlier conclusion.
Therefore $(U_{1},U_{2})$ is a vertical separation.
\end{proof}

For $i=1,2$, we let $F_{i}$ be $\cl(U_{i})$.
Recall that $R_{i}$ is the edge-set of $C_{i}$.

\begin{claim}
$F_{1}\cap F_{2}=R_{1}\cap R_{2}$.
\end{claim}

\begin{proof}
Let $e$ be an edge that joins vertices $u$ and $v$.
First assume that $e$ is in $R_{1}\cap R_{2}$.
Then $u$ and $v$ are in $S$.
This means there is no $S$\dash avoiding path containing $e$ with $a_{1}$ as a terminal vertex.
Hence $e$ is not in $U_{1}$ so it is in $U_{2}$.
However, $a_{1}$ is adjacent to $u$ and $v$, and the edges $a_{1}u$ and $a_{1}v$ are in $U_{1}$, so $e$ is in $\cl(U_{1})$.
Thus $e$ is in $F_{1}\cap F_{2}$ and we have shown that $R_{1}\cap R_{2}\subseteq F_{1}\cap F_{2}$.

For the other direction, assume that $e$ is in $F_{1}\cap F_{2}$.
First assume that $e$ is in $U_{1}$.
Let $P$ be an $S$\dash avoiding path with $a_{1}$ as a terminal vertex such that $e$ is in $P$.
We can assume that either $u$ or $v$ is a terminal vertex of $P$.

Since $e$ is in $U_{1}\cap \cl(U_{2})$ we can let $D$ be a cycle such that $e$ is in $D$, and every other edge of $D$ is in $U_{2}$.
Thus both $u$ and $v$ are incident with edges in $U_{2}$.
Let $e'$ be an edge incident with $u$ that is in $U_{2}$.
Assume for a contradiction that $u$ is not in $S$.
If $u$ is a terminal vertex of $P$ then we obtain a new path by adding $e'$ to the end of $P$.
No internal vertex of this new path is in $S$, so it implies that $e'$ is in $U_{1}$, a contradiction to $e'$ being in $U_{2}$.
Therefore $u$ is not a terminal vertex of $P$.
Since $P$ contains $e$, it follows that $u$ is an internal vertex of $P$, so $v$ is a terminal vertex of $P$.
In this case we can obtain a new path from $P$ by replacing the edge $e$ with $e'$.
Again we see that $e'$ is in $U_{1}$ and we have a contradiction.
Therefore $u$ is in $S$, and by symmetry, so is $v$.
Hence $e$ joins two vertices of $S$, and is thus in $R_{1}\cap R_{2}$.

We must also consider the case that $e$ is in $U_{2}\cap\cl(U_{1})$.
Let $D$ be a cycle that contains $e$, where every other edge of $D$ is in $U_{1}$.
Let $x$ be an edge of $D-e$ that is incident with $u$.
Thus $x$ is in $U_{1}$.
Let $P$ be an $S$\dash avoiding path containing $x$ and $a_{1}$ as a terminal vertex.
If $u$ is not in $S$, then we can either extend $P$ by adding the edge $e$, or replacing $x$ in $P$ with $e$.
In either case, the new path shows that $e$ is in $U_{1}$, a contradiction.
Therefore $u$, and by symmetry $v$, is in $S$, so we again see that $e$ is in $R_{1}\cap R_{2}$.
Hence $F_{1}\cap F_{2}\subseteq R_{1}\cap R_{2}$ and the claim is proved.
\end{proof}
    
Recall that $a_{1}$ is in the clique $C_{1}$.
Every edge incident with $a_{1}$ is in $U_{1}$.
As every edge of $C_{1}-a_{1}$ is spanned by two such edges, it follows that $R_{1}$ is contained in $F_{1}$.
We must also show that $R_{2}$ is spanned by $U_{2}$.
Let $e$ be an edge in $R_{2}$ and assume that it is not in $F_{2}$.
In particular, this means that $e$ is not in $U_{2}$, so it is in $U_{1}$.
Let $P$ be an $S$\dash avoiding path containing $e$ that has $a_{1}$ as a terminal vertex.
Let $u$ be the first internal vertex of $P$ that is incident with $e$.
Then $u$ is not in $S$, as $P$ is $S$\dash avoiding.
But $u$ is in $C_{2}$, since $e$ is in $R_{2}$.
Thus $u$ is in $C_{2}-C_{1}$, and the subpath of $P$ from $a_{1}$ to $u$ is an $S$\dash avoiding path from a vertex of $C_{1}-C_{2}$ to a vertex of $C_{2}-C_{1}$.
This is a contradiction, as $C_{1}$ and $C_{2}$ form a separating pair.
This shows that $R_{2}$ is contained in $F_{2}$, as claimed.

Now we can complete the proof that $R_{1}$ and $R_{2}$ are adjacent in $R(M)$ by showing that $(F_{1},F_{2})$ is a modular cover.
Assume that $F_{1}$ is not a modular flat, so that by utilising Proposition \ref{mod-flat-new-def} we can let $F$ be an arbitrary flat of $M$ that is disjoint from $F_{1}$ such that some circuit $C\subseteq F\cup F_{1}$ contains elements of both $F$ and $F_{1}$.
If each connected component of $G[F]$ shares at most one vertex with $G[F_{1}]$, then no such cycle can exist.
Therefore we let $u$ and $v$ be distinct vertices from the same connected component of $G[F]$ so that both of $u$ and $v$ are incident with edges in $F_{1}$.
Since $u$ is incident with an edge in $F_{1}$, it is incident with an edge in $U_{1}$.
Let $e$ be such an edge, and let $P$ be a shortest-possible $S$\dash avoiding path that contains $e$ and has $a_{1}$ as a terminal vertex.
Let $f$ be an edge of $F$ that is incident with $u$.
If $u$ is not in $S$, then extending $P$ by adding $f$ shows that $f$ is in $U_{1}$, a contradiction.
Therefore $u$, and by symmetry $v$, is in $S$.
This means that there is an edge $g$ of $R_{1}$ that joins $u$ and $v$.
Thus $g$ is in $R_{1}\subseteq F_{1}$.
But there is a path of $G[F]$ that joins $u$ to $v$, so $g$ is in $\cl(F) = F$, and we have contradicted $F\cap F_{1}=\emptyset$.
Therefore $F_{1}$ is a modular flat.
Almost exactly the same argument shows that $F_{2}$ is a modular flat.
\end{proof}

The previous result shows that when $G$ is a $2$\dash connected chordal graph, $C_{R}(G)$ is isomorphic to the rotunda graph of a supersolvable saturated matroid.
In fact, this is true even when $G$ is not $2$\dash connected, as we now show.
First we make a simple observation.
Recall from Proposition \ref{supersolvable-components} that a matroid is supersolvable and saturated if and only if all its components are.

\begin{proposition}
\label{conn-components}
Let $G$ be a chordal graph with connected components $H_{1},\ldots, H_{k}$.
Then $C_{R}(G)$ is the disjoint union of $C_{R}(H_{1}),\ldots, C_{R}(H_{k})$.
Similarly, if $M$ is a supersolvable saturated matroid with connected components $N_{1},\ldots, N_{k}$, then $R(M)$ is the disjoint union of $R(N_{1}),\ldots, R(N_{k})$.
\end{proposition}

\begin{proof}
Maximal cliques in different components of $G$ cannot be adjacent in $C_{R}(G)$ because they have no vertices in common.
Similarly, rotunda from different components of $M$ are not adjacent in $R(M)$ because they have empty intersection.
The result follows.
\end{proof}

\begin{lemma}
\label{RCgraph-is-ROTgraph}
Let $G$ be a chordal graph.
There is a supersolvable saturated matroid $M$ such that $C_{R}(G)$ is isomorphic to $R(M)$.
\end{lemma}

\begin{proof}
Let $H_{1},\ldots, H_{k}$ be the connected components of $G$.
If each $C_{R}(H_{i})$ is isomorphic to $R(N_{i})$ for some supersolvable saturated $N_{i}$, then Proposition \ref{conn-components} implies that $C_{R}(G)$ is isomorphic to $R(N_{1}\oplus\cdots\oplus N_{k})$.
In other words, it suffices to prove the lemma when $G$ is connected.
In this case, we will prove that $C_{R}(G)$ is isomorphic to $C_{R}(G')$, where $G'$ is a $2$\dash connected chordal graph.
Then Proposition \ref{2-con-same} shows that $C_{R}(G')=R(M(G'))$, and $M(G')$ is supersolvable and saturated by Corollary \ref{graphic-implies-saturation} and Proposition \ref{chordal-iff-supersolv}, so the result will follow.

If $G$ is $2$\dash connected, then there is nothing left to prove, so let $v_{1},v_{2},\ldots, v_{m}$ be the cut-vertices of $G$.
We produce $G'$ by introducing new vertices $v_{1}',v_{2}',\ldots, v_{m}'$ and for each $i$ making $v_{i}'$ adjacent to $v_{i}$ and all of the neighbours of $v_{i}$ in $G$.

\begin{claim}
\label{VS-lemma-claim1}
$G'$ is $2$\dash connected.
\end{claim}

\begin{proof}
Certainly $G'$ is connected.
Assume that $v$ is a cut-vertex of $G'$.
Note that for each $i$, the graph produced from $G'$ by deleting $v_{i}'$ is obtained from $G$ by adding $m-1$ new vertices and making each of them adjacent to at least one vertex in $G$.
Since $G$ is connected it follows that $G'-v_{i}'$ is connected.
Thus no vertex $v_{i}'$ is a cut-vertex of $G'$ so $v$ is not equal to $v_{i}'$ for any $i$.
Now $v$ is a vertex of $G$.
If $v\notin \{v_{1},v_{2},\ldots, v_{m}\}$ then $G-v$ is connected and $G'-v$ is obtained from the connected graph $G-v$ by adding $m$ new vertices and making each of them adjacent to at least one vertex in $G-v$.
Thus $G'-v$ is connected, which is a contradiction.
Therefore $v=v_{i}$ for some $i$.
But in $G'$ the vertices $v_{i}$ and $v_{i}'$ are adjacent to exactly the same vertices.
Therefore $G'-v_{i}$ is obtained from $G'-v_{i}'$ by relabelling $v_{i}'$ as $v_{i}$.
This means that $G'-v_{i}$ is connected, and we have a contradiction.
\end{proof}

\begin{claim}
\label{VS-lemma-claim2}
$G'$ is chordal.
\end{claim}

\begin{proof}
We rely on Proposition \ref{chordal-characterisation}.
Let $u_{1},u_{2},\ldots, u_{n}$ be a perfect elimination order of $G$.
We produce an ordering of the vertices of $G'$ by inserting each $v_{i}'$ into the order $u_{1},u_{2},\ldots, u_{n}$ immediately after $v_{i}$.
It is easy to verify that this produces a perfect elimination order for $G'$ and the result follows.
\end{proof}

We can complete the proof by showing that $C_{R}(G)$ is isomorphic to $C_{R}(G')$.
It is clear that any maximal clique of $G'$ contains one of the vertices $\{v_{i},v_{i}'\}$ if and only if it contains both.
Now we can easily verify that there is a bijective correspondence between the maximal cliques of $G$ and the maximal cliques of $G'$.
If $C$ is a maximal clique of $G$, then we obtain the corresponding maximal clique of $G'$ by adding each vertex $v_{i}'$ such that $v_{i}$ is in $C$.

Let $C_{1}$ and $C_{2}$ be distinct maximal cliques of $G$, and let $C_{1}'$ and $C_{2}'$ be the corresponding maximal cliques of $G'$.
We will prove that $C_{1}$ and $C_{2}$ are adjacent in $C_{R}(G)$ if and only if $C_{1}'$ and $C_{2}'$ are adjacent in $C_{R}(G')$.
First note that $C_{1}\cap C_{2}$ is non-empty if and only if $C_{1}'\cap C_{2}'$ is non-empty.

If $C_{1}$ and $C_{2}$ are not adjacent in $C_{R}(G)$, then either $C_{1}\cap C_{2}=\emptyset$, or $P$ is a $(C_{1}\cap C_{2})$\dash avoiding path of $G$ from a vertex of $C_{1}-C_{2}$ to a vertex in $C_{2}-C_{1}$.
In the first case $C_{1}'\cap C_{2}'=\emptyset$.
In the second case, it is obvious that $P$ is a $(C_{1}'\cap C_{2}')$\dash avoiding path of $G'$.
In either case $C_{1}'$ and $C_{2}'$ are not adjacent in $C_{R}(G')$.

Next assume that $C_{1}'$ and $C_{2}'$ are not adjacent in $C_{R}(G')$.
If $C_{1}'\cap C_{2}'=\emptyset$ then $C_{1}\cap C_{2}=\emptyset$ so we have nothing left to prove.
Therefore we will assume that $P$ is a $(C_{1}'\cap C_{2}')$\dash avoiding path in $G'$, and that $P$ joins a vertex in $C_{1}'-C_{2}'$ to a vertex in $C_{2}'-C_{1}'$.
If a vertex $v_{i}'$ appears anywhere in $P$, then we may replace it with $v_{i}$, since these two vertices have the same neighbourhoods.
Note that the resulting path is still $(C_{1}'\cap C_{2}')$\dash avoiding, and still joins a vertex of $C_{1}'-C_{2}'$ to a vertex of $C_{2}'-C_{1}'$.
Thus we can assume that $P$ is a path of $G$, and is consequently a $(C_{1}\cap C_{2})$\dash avoiding path of $G$ from a vertex of $C_{1}-C_{2}$ to a vertex of $C_{2}-C_{1}$.
This shows that $C_{1}$ and $C_{2}$ are not adjacent in $C_{R}(G)$ so the proof is complete.
\end{proof}

We have established that every reduced clique graph is isomorphic to a rotunda graph.
Next we start moving towards proving the converse.

\begin{definition}
\label{compliant-graph}
Let $M$ be a connected supersolvable and saturated matroid, and let $G$ be a $2$\dash connected chordal graph.
Assume that $\theta$ is a function from $E(M)$ to the powerset of $V(G)$.
If $U$ is a subset of vertices in $G$, then let $\theta^{-1}(U)$ be $\{x\in E(M)\colon \theta(x)\subseteq U\}$.
For any subset $R\subseteq E(M)$, let $\theta(R)$ stand for $\cup_{x\in R}\theta(x)$.
Thus we can think of $\theta$ as being a function from $\mathcal{P}(E(M))$ to $\mathcal{P}(V(G))$ such that $R\subseteq R'$ if and only if $\theta(R)\subseteq \theta(R')$.
Assume that the following properties hold:
\begin{enumerate}[label = \textup{(\roman*)}]
\item $|\theta(x)|= 2$ for every $x\in E(M)$, and
\item for any vertex $v\in V(G)$ there exists exactly one element $x\in E(M)$ such that $v$ is in $\theta(x)$.
\item if $R$ is a non-empty round flat of $M$, then $\theta(R)$ is a clique,
\item if $F$ is a modular flat of $M$ and $U$ is a union of connected components of $G-\theta(F)$, then $F\cup \theta^{-1}(U)$ is a modular flat of $M$, and
\item the restriction of $\theta$ to $\mathcal{R}(M)$ is a bijection from $\mathcal{R}(M)$ to the maximal cliques of $C_{R}(G)$ and this bijection is an isomorphism between $R(M)$ and $C_{R}(G)$.
\end{enumerate}
If all these conditions hold, then we will say that $(G,\theta)$ is \emph{compliant} with $M$.
\end{definition}
 
\begin{lemma}
\label{ROTgraph-is-RCgraph}
Let $M$ be a connected supersolvable and saturated matroid.
There exists a $2$\dash connected chordal graph $G$ and a function $\theta\colon E(M)\to\mathcal{P}(V(G))$ such that $(G,\theta)$ is compliant with $M$.
\end{lemma}

\begin{proof}
The proof is a straightforward induction, although the technical details are require some work.
If $M$ has rank at most one, then we can simply make $G$ a clique of the appropriate size.
Now we are going to choose $C^{*}$ to be the complement of a modular hyperplane, $H$.
Then inductively $M|H$ has a compliant graph.
The intersection of $H$ with $\cl(C^{*})$ is a round flat, and therefore corresponds to a clique.
We create a new maximal clique by adding new vertices and making them adjacent to each other and to the clique corresponding to $H\cap\cl(C^{*})$.
The rest of the proof involves nothing more than checking that this construction does indeed satisfy the conditions for compliance.

To implement this strategy, we let $M$ be a supersolvable saturated matroid.
Assume $r(M)\leq 1$.
We can easily see that the only rotunda of $M$ is $E(M)$ itself.
We let $G$ be isomorphic to $K_{2|E(M)|}$, and we consider an arbitrary partition of $V(G)$ into blocks of size two.
We then set $\theta$ to be an arbitrary bijection from $E(M)$ to the blocks of the partition.
It is not hard to verify that $(G,\theta)$ is compliant with $M$.
Therefore we assume that $r(M)> 1$.

Let $H$ be a modular hyperplane of $M$ such that $M|H$ is supersolvable.
Then $M|H$ is also saturated.
Proposition \ref{mod-hyp-connected} says that $M|H$ is connected.
Therefore we can apply the obvious inductive hypothesis and let $G'$ be a $2$\dash connected chordal graph with a function $\theta'\colon H \to \mathcal{P}(V(G'))$ such that $(G',\theta')$ is compliant with $M|H$.

Let $C^{*}$ be the complementary cocircuit of $H$, and let $R$ be the closure of $C^{*}$.
Proposition \ref{cocircuit-closure-rotunda} says that $R$ is a rotunda, and furthermore it is the only rotunda of $M$ that is not contained in $H$.
Certainly $C^{*}$ is non-empty, and $r(H)=r(M)-1>0$, so $H$ is non-empty also.
But $(H,C^{*})$ is not a separation of $M$, since $M$ is connected.
As $H$ is modular, we deduce that
\[
r(R\cap H) = r(\cl(C^{*})\cap H) = r(C^{*})+r(H)-r(M)>0.
\]
Therefore $R\cap H$ is non-empty and Proposition \ref{mod-hyp-round-intersect} tells us that $R\cap H$ is round.

Let $W$ be $\theta'(R\cap H)$.
Since $(G',\theta')$ is compliant with $M|H$, we see that $W$ is the set of vertices of a clique in $G'$.
Note also that $|W|= 2|R\cap H|\geq 2$.
We produce $G$ from $G'$ by adding $Y$, a set of $2|C^{*}|$ new vertices, and making each of them adjacent to all the vertices of $W$.
Note that $W\cup Y$ is a maximal clique of $G$ and $G'=G-Y$.
Because $W$ has at least two vertices it is easy to see that $G$ is $2$\dash connected.
The neighbours of any vertex in $Y$ form a clique in $G$.
Therefore we can construct a perfect elimination order for $G$ by prepending the vertices of $Y$ to a perfect elimination order for $G'$.
It follows that $G$ is chordal.

Consider an arbitrary partition of $Y$ into pairs of vertices, and let $\phi$ be an arbitrary bijection from $C^{*}$ to the blocks of this partition.
Then we define $\theta$ to be the union of $\theta'$ and $\phi$.
Note that $|\theta(x)|=2$ for any $x\in E(M)$, and for any vertex $v$ of $G$, there is exactly one element $x\in E(M)$ such that $v$ is in $\theta(x)$.
Therefore the remainder of the proof consists in showing that $\theta$ satisfies conditions (iii), (iv), and (v) in Definition \ref{compliant-graph}.

Proposition \ref{prop2} tells us that if $Z$ is a round flat of $M$ then either $Z\subseteq H$ or $Z\subseteq R$.
In the former case, $Z$ is a round flat of $M|H$, and $\theta(Z)=\theta'(Z)$ is a clique of $G$, since $\theta'$ satisfies (iii).
In the latter case $\theta(Z)$ is a subset of $W\cup Y$, and again $\theta(Z)$ is a clique of $G$.
So condition (iii) holds for $(G,\theta)$.

\begin{claim}
Condition \textup{(iv)} in \textup{Definition \ref{compliant-graph}} holds for $(G,\theta)$.
\end{claim}

\begin{proof}
Let $F$ be a modular flat of $M$, and let $U$ be a union of connected components of $G-\theta(F)$.
Let $D$ be $F\cup \theta^{-1}(U)$.
Thus our aim is to show that $D$ is a flat of $M$.
Assume that $U$ is the empty union.
In this case $D = F\cup\theta^{-1}(U) = F$ and since $F$ is a modular flat there is nothing left to prove.
Therefore we assume that $U$ contains at least one connected component of $G-\theta(F)$.

Assume that $D$ is disjoint with $C^{*}$.
This means that $\theta(F)\cup U$ is disjoint with $Y$.
Thus $F$ is a modular flat of $M|H$.
If $U$ is not a union of connected components in $G'-\theta'(F)$ then there is a connected component of this graph that contains vertices $u\in U$ and $v\notin U$.
There is a path of $G'-\theta'(F) = G-(\theta(F)\cup Y)$ from $u$ to $v$.
Hence $u$ and $v$ are in the same component of $G-\theta(F)$.
This contradicts the fact that $U$ is a union of components in this graph.
Hence $U$ is a union of components of $G'-\theta'(F)$, so we can apply the inductive assumption and see that $D=F\cup(\theta')^{-1}(U) = F\cup\theta^{-1}(U)$ is a modular flat of $M|H$.
Therefore $D$ is a modular flat of $M$ and we are done.
Hence we assume that $D$ contains at least one element of $C^{*}$.

Now $\theta(F)\cup U$ contains at least two vertices from $Y$.
Since any such vertex is adjacent to every vertex in $W\cup Y$, and $U$ is a non-empty union of connected components, it now follows that $\theta(F)\cup U$ contains $W\cup Y$.
Note that $D$ contains $\theta^{-1}(W\cup Y) = R=\cl(C^{*})$.

Assume that $U-Y$ is not a union of connected components in $G'-\theta(F\cap H)$.
Then there is a connected component of $G'-\theta(F\cap H)$ that contains vertices $u\in U-Y$ and $v\notin U-Y$.
There is a path from $u$ to $v$ in $G'-\theta(F\cap H) = G - (\theta(F)\cup Y)$.
Thus $u$ and $v$ are in the same connected component of $G-\theta(F)$.
This means that $u$ and $v$ are both in $U$, since $U$ is a union of connected components in this graph.
Since $v$ is not in $U-Y$ this means that $v$ is in $Y$.
But this is impossible, since $v$ is a vertex of $G'$, which is equal to $G-Y$.
This shows that $U-Y$ is a union of connected components in $G'-\theta(F\cap H)$.

We note that $F\cap H$ is a modular flat of $M|H$ since both $F$ and $H$ are modular in $M$.
The inductive hypothesis now tells us that \[(F\cap H)\cup \theta^{-1}(U-Y)\] is a modular flat of $M|H$.
Let this flat be $D'$.
Note that because $C^{*}\subseteq D$ we have
\[
D = F\cup \theta^{-1}(U) = (F\cap H)\cup\theta^{-1}(U-Y)\cup C^{*} = D'\cup C^{*}.
\]
Let $P$ be the union $\cup P_{H}(x,y)$, where $x$ and $y$ range over all distinct rank-one flats contained in $C^{*}$.
Thus $P$ is a subset of $R\cap H\subseteq D\cap H = D'$.
Note that
\[\cl(D) = (\cl(D) \cap H) \cup C^{*} = (\cl(D'\cup C^{*})\cap H)\cup C^{*}.\]
Now we apply Proposition \ref{mod-hyp-sep-extends}.
Since $P\subseteq D'\subseteq H$ we see that
\[(\cl(D'\cup C^{*})\cap H)\cup C^{*}= \cl(D')\cup C^{*} = D'\cup C^{*} = D.\]
Thus $D$ is a flat of $M$.

Assume that $D$ is not a modular flat of $M$, and let $F'$ be a flat of $M$ that is disjoint with $D$, chosen so that $C\subseteq D\cup F'$ is a circuit that contains elements of both $D$ and $F'$.
Choose $C$ so that $|C\cap C^{*}|$ is as small as possible.
Exactly as in the proof of Proposition \ref{mod-hyp-vert-sep} we can prove that $C\cap C^{*}$ contains distinct elements $x$ and $y$.
We choose $p$ to be an element in $P_{H}(x,y)$, and we perform strong circuit elimination on $C$ and $\{x,y,p\}$.
In this way we find a circuit contained in $D\cup F'$ that contains elements of both sets, and contains fewer elements of $C^{*}$ than $C$.
This contradiction shows that $D$ is a modular flat of $M$ so condition (iv) holds.
\end{proof}

\begin{claim}
\label{RC-ROT-claim1}
The restriction of $\theta$ to $\mathcal{R}(M)$ is a bijection between $\mathcal{R}(M)$ and the maximal cliques of $G$.
\end{claim}

\begin{proof}
The inductive hypothesis means that $\theta'$ induces a bijection between the rotunda of $M|H$ and the maximal cliques of $G'$.
First assume that $R\cap H$ is a rotunda of $M|H$, so that $W=\theta'(R\cap H)$ is a maximal clique of $G'$.
Now Proposition \ref{mod-hyp-rot-graph} shows that the rotunda of $M$ are the rotunda of $M|H$, except that $R\cap H$ has been replaced by $R$.
It is easy to see that he maximal cliques of $G$ are the maximal cliques of $G'$, except that $W$ has been replaced by $W\cup Y$.
We observe that $\theta(R) = W\cup Y$ and now it follows that $\theta|_{\mathcal{R}(M)}$ is a bijection between the rotunda of $M$ and the maximal cliques of $G$.

Next we assume that $R\cap H$ is not a rotunda of $M|H$.
Proposition \ref{mod-hyp-rot-graph} implies that every rotunda of $M|H$ is also a rotunda of $M$.
Furthermore $R$ is the only rotunda of $M$ that is not a rotunda of $M|H$.
Because $R\cap H$ is not a rotunda of $M|H$, we can let $Z$ be a rotunda of $M|H$ that properly contains $R\cap H$.
Now $W=\theta'(R\cap H)$ is properly contained in $\theta'(Z)$.
Since $Z$ is round, we see that $\theta(Z)=\theta'(Z)$ is a clique that properly contains $W$.
Therefore $W$ is not a maximal clique of $G'$.
Now it is easy to see that every maximal clique of $G'$ is a maximal clique of $G$, and that $W\cup Y$ is the only maximal clique of $G$ that is not a maximal clique of $G'$.
The claim follows.
\end{proof}

We can complete the proof of Lemma \ref{ROTgraph-is-RCgraph} by proving that the restriction of $\theta$ to $\mathcal{R}(M)$ is an isomorphism from $R(M)$ to $C_{R}(G)$.
Let $Z$ and $Z'$ be distinct rotunda of $M$.
We will show that they are adjacent in $R(M)$ if and only if $\theta(Z)$ and $\theta(Z')$ are adjacent in $C_{R}(G)$.

\noindent\textbf{Case 1.} \emph{Neither $Z$ nor $Z'$ is equal to $R$}.
In this case both $Z$ and $Z'$ are rotunda of $M|H$, and $\theta(Z)$ and $\theta(Z')$ are maximal cliques of $G'$.
Assume that $\theta(Z)$ and $\theta(Z')$ are adjacent in $C_{R}(G)$.
Then these maximal cliques have at least one vertex in common, and there is no $(\theta(Z)\cap\theta(Z'))$\dash avoiding path in $G$ from a vertex of $\theta(Z)-\theta(Z')$ to a vertex of $\theta(Z')-\theta(Z)$.
Exactly the same statements apply to $\theta'(Z)$ and $\theta'(Z')$ in $G'$, so $\theta'(Z)$ and $\theta'(Z')$ are adjacent in $C_{R}(G')$.
The inductive assumption implies that $Z$ and $Z'$ are adjacent in $R(M|H)$.
Proposition \ref{mod-hyp-rot-graph} now implies that they are also adjacent in $R(M)$.

For the converse, assume that $Z$ and $Z'$ are adjacent in $R(M)$, and let $(F,F')$ be a modular cover of $M$ that certifies the adjacency.
We assume that $Z\subseteq F$ and $Z'\subseteq F'$.
Let us assume that both $F\cap C^{*}$ and $F'\cap C^{*}$ are non-empty.
No element of $F\cap C^{*}$ is in $F'$, because any such element would be in $F\cap F'=Z\cap Z'$, and this is not possible since $Z$ and $Z'$ are subsets of $H=E(M)-C^{*}$.
Symmetrically, no element of $F'\cap C^{*}$ is in $F$.
So $F\cap R$ does not contain any element of $F'\cap C^{*}$ and $F'\cap R$ does not contain any element of $F\cap C^{*}$.
This shows that $(F\cap R,F'\cap R)$ is a vertical cover of $R$, which is impossible as $R$ is a round flat.
Therefore either $F\cap C^{*}=\emptyset$ or $F'\cap C^{*} = \emptyset$.
We assume the latter, so $C^{*}$ is a subset of $F$ and $F'$ is a subset of $H$.

Proposition \ref{round-separations} says that $F'$ does not contain $Z$.
It therefore does not contain $H$, so we can apply Proposition \ref{prop3} and deduce that \[(F\cap H,F'\cap H) = (F\cap H,F')\] is a modular cover of $M|H$.
Note that
\[
 (F\cap H)\cap F' = F\cap F' = Z\cap Z'
\]
so $Z$ and $Z'$ are adjacent in $R(M|H)$.
By the inductive hypothesis, $\theta'(Z)=\theta(Z)$ and $\theta'(Z')=\theta(Z')$ are adjacent in $C_{R}(G')$.
Since $Z$ and $Z'$ are rotunda of $M$, neither is equal to $R\cap H$, which is properly contained in $R$.
Therefore neither $\theta(Z)$ nor $\theta(Z')$ is equal to $W$.
Because $\theta(Z)$ and $\theta(Z')$ are adjacent in $C_{R}(G')$ they have a non-empty intersection.

Assume that $\theta(Z)$ and $\theta(Z')$ are not adjacent in $C_{R}(G)$.
Let $P$ be a $(\theta(Z)\cap \theta(Z'))$\dash avoiding path of $G$ from a vertex $a\in \theta(Z)-\theta(Z')$ to a vertex $b\in \theta(Z')-\theta(Z)$.
Because no such path can exist in $G'=G-Y$, it follows that $P$ contains a vertex in $Y$.
Let $y$ and $y'$, respectively, be the first and last vertices of $P$ that are in $Y$.
Note that $y$ and $y'$ are not equal to $a$ or $b$, which are vertices of $G'$.
Let $w$ be the neighbour of $y$ in the subpath of $P$ from $y$ to $a$.
Similarly let $w'$ be the neighbour of $y'$ in the subpath from $y'$ to $b$.
Because $w$ and $w'$ are adjacent to vertices in $Y$, but are not in $Y$, they must be in $W$.
Thus $w$ and $w'$ are adjacent, so there is a path of $G'$ from $a$ to $b$ that avoids any vertex in $\theta(Z)\cap \theta(Z')$.
This is a contradiction, so we conclude that $\theta(Z)$ and $\theta(Z')$ are adjacent in $R(M)$.

We have now completed the case that neither $Z$ nor $Z'$ is equal to $R$.

\noindent\textbf{Case 2.} \emph{One of $Z$ and $Z'$ is equal to $R$.}
We let $Z$ be a rotunda of $M$ that is distinct from $R$, and we will prove that $Z$ and $R$ are adjacent in $R(M)$ if and only if $\theta(Z)$ and $\theta(R)=W\cup Y$ are adjacent in $C_{R}(G)$.
Observe that $Z$ is contained in $H$ by Proposition \ref{cocircuit-closure-rotunda}.

First assume that $\theta(Z)$ and $\theta(R)$ are adjacent in $C_{R}(G)$.
Because $\theta$ sends distinct elements of $E(M)$ to distinct pairs of vertices, it cannot be the case that $Z\cap R=\emptyset$, or else $\theta(Z)$ and $\theta(R)$ would have no vertices in common, contradicting their adjacency in $C_{R}(G)$.
Thus $Z$ and $R$ are non-disjoint.

Assume $Z$ contains $R\cap H$.
If $C^{*}$ is spanning in $M$, then $R\cap H = H$, so $\theta(H)=W$ is a clique.
In this case $G=W\cup Y$ is a clique, but we have assumed that $M$ has at least two distinct rotunda, so $G$ has at least two distinct maximal cliques by \ref{RC-ROT-claim1}.
Thus $C^{*}$ is not spanning.
Proposition \ref{supsol-mod-sep} says that $(H,R)$ is a modular cover of $M$.
Now $R\cap H = R\cap Z$ so $(H,R)$ certifies that $Z$ and $R$ are adjacent in $R(M)$ and we have nothing left to prove.
Therefore we will assume that $Z$ does not contain $R\cap H$.
Hence $Z\cap R$ is a proper and non-empty subset of $R\cap H$.
It follows that $\theta(Z)$ contains some, but not all, of the vertices of $W$.

By Proposition \ref{mod-hyp-round-intersect} we know that $R\cap H$ is a round flat of $M|H$.
Let $Z_{0}$ be a rotunda of $M|H$ that contains $R\cap H$.
Thus $Z_{0}$ is not equal to $Z$, but it may be equal to $R\cap H$.
Now $\theta'(Z_{0})=\theta(Z_{0})$ is a maximal clique of $G'$ that contains $W$.
Assume that $\theta(Z)$ and $\theta(Z_{0})$ are not adjacent in $C_{R}(G')$.
Because these cliques have at least one vertex of $W$ in common, we can let $P$ be a $(\theta(Z)\cap\theta(Z_{0}))$\dash avoiding path of $G'$ from a vertex $a\in \theta(Z)-\theta(Z_{0})$ to a vertex $b\in \theta(Z_{0})-\theta(Z)$.
Note that $P$ contains no vertex of $\theta(Z)\cap\theta(R)$.
But $P$ is also a path of $G$, and $b$ is adjacent to any vertex of $W-\theta(Z)$.
Thus, if necessary, we can adjoin an edge to $P$ from $b$ to a vertex of $W-\theta(Z)$, and certify that $\theta(Z)$ and $\theta(R)$ are not adjacent in $C_{R}(G)$, contrary to hypothesis.
Therefore $\theta(Z)$ and $\theta(Z_{0})$ are adjacent in $C_{R}(G')$, so by induction $Z$ and $Z_{0}$ are adjacent in $R(M|H)$.

Because $Z_{0}$ contains $R\cap H$, the intersection of $Z$ and $Z_{0}$ contains $Z\cap R$.
Assume this containment is proper, and let $e$ be an element of $Z\cap Z_{0}$ that is not in $Z\cap R$.
Let $v$ be a vertex in $\theta(e)$.
Thus $v$ is in $\theta(Z)-\theta(R)$.
Choose $w$, an arbitrary vertex in $W-\theta(Z)$.
Because $v$ is in $\theta(Z_{0})$, which contains $W$, it follows that $v$ and $w$ are adjacent.
Since $w$ is in $\theta(R)-\theta(Z)$, we now see that $\theta(Z)$ and $\theta(R)$ are not adjacent in $C_{R}(G)$, contrary to hypothesis.
We conclude that $Z\cap Z_{0} = Z\cap R$.

Since $Z$ and $Z_{0}$ are adjacent in $R(M|H)$, we can let $(F,F')$ be a modular cover of $M|H$ that certifies this adjacency, where $Z\subseteq F$ and $Z_{0}\subseteq F'$.
Because $Z_{0}$ contains $R\cap H$, it follows that $F'$ contains $\cup P_{H}(x,y)$, where $x$ and $y$ range over distinct rank-one flats contained in $C^{*}$.
Proposition \ref{mod-hyp-vert-sep} says that $(F,F'\cup C^{*})$ is a modular cover of $M$.
Certainly $Z\subseteq F$ and $Z_{0}\subseteq F'\cup C^{*}$.
Furthermore, 
\[
F\cap (F'\cup C^{*})=F\cap F' = Z\cap Z_{0} = Z\cap R.
\]
Thus $(F,F'\cup C^{*})$ certifies that $Z$ and $R$ are adjacent in $R(M)$, exactly as desired.

For the converse, we assume that $Z$ and $R$ are adjacent in $R(M)$.
Thus $Z\cap R$ is non-empty.
Assume that $Z$ contains $R\cap H$.
Then $\theta(Z)$ contains $\theta(R\cap H) = W$, so $\theta(Z)\cap \theta(R) = W$.
In $G-W$ there is no path from a vertex of $\theta(R)-\theta(Z) = Y$ to a vertex not in $Y$, and in particular there is no path to a vertex in $\theta(Z)-\theta(R)$.
So in this case $\theta(Z)$ and $\theta(R)$ are adjacent in $C_{R}(G)$ and we have nothing left to prove.
Therefore we will assume that $Z$ does not contain $R\cap H$.
Hence $Z\cap R$ is a non-empty proper subset of $R\cap H$.
Since $R\cap H$ is a round flat of $M|H$ by Proposition \ref{mod-hyp-round-intersect}, we can let $Z_{0}$ be a rotunda of $M|H$ that contains $R\cap H$.
Thus $Z_{0}$ may be equal to $R\cap H$, but it is not equal to $Z$.

Let $(F,F')$ be a modular cover of $M$ that certifies the adjacency of $R$ and $Z$ in $R(M)$, where $R\subseteq F$ and $Z\subseteq F'$.
Because $F\cap F' = R\cap Z$ and $Z$ is contained in $H$ it follows that $F'$ is contained in $H$.
If $F'=H$, then $F\cap F'$ contains $R\cap H$, which properly contains $Z\cap R$.
This contradicts $F\cap F' = R\cap Z$, so $F'$ does not contain $H$.
By applying Proposition \ref{prop3}, we see that $(F\cap H,F'\cap H)=(F\cap H, F')$ is a modular cover of $M|H$.

Because $Z_{0}$ is round, one of $F\cap Z_{0}$ and $F'\cap Z_{0}$ is not a proper flat of $M|Z_{0}$.
That is, $Z_{0}$ is contained in either $F$ or $F'$.
Assume $Z_{0}$ is contained in $F'$.
Then $R\cap H\subseteq Z_{0}\subseteq F'$ and $R\subseteq F$ so $F\cap F' $ contains $R\cap H$.
This is a contradiction as $F\cap F' = R\cap Z$, which is a non-empty proper subset of $R\cap H$.
Therefore $Z_{0}$ is contained in $F$.
We observe that
\[
(F\cap H)\cap F' = (F\cap F')\cap H = (R\cap Z)\cap H = R\cap Z = F\cap F' \supseteq Z_{0}\cap Z.
\]
Assume that $F\cap F'$ properly contains $Z_{0}\cap Z$ and let $e$ be an element of $(F\cap F')-(Z_{0}\cap Z)$.
Since $F\cap F' = R\cap Z$ it follows that $e$ is in $Z$.
But we also have
\[
e\in F\cap F' = R\cap Z\subset R\cap H\subseteq Z_{0}.
\]
Thus $e$ is in $Z_{0}\cap Z$ after all and we have a contradiction.
Thus $(F\cap H)\cap F' = Z_{0}\cap Z = R\cap Z$ and the modular cover $(F\cap H,F')$ of $M|H$ certifies that $Z_{0}$ and $Z$ are adjacent in $R(M|H)$.
Induction now tells us that $\theta(Z_{0})$ and $\theta(Z)$ are adjacent in $C_{R}(G')$.

Assume that $\theta(Z)$ and $\theta(R)=W\cup Y$ are not adjacent in $C_{R}(G)$.
These cliques certainly have common vertices, so we can let $P$ be a path from $a\in \theta(Z)-\theta(R)$ to $b\in\theta(R)-\theta(Z)$ such that $P$ contains no vertex of $\theta(Z)\cap\theta(R) = \theta(Z)\cap\theta(Z_{0})$.
If $P$ is a path of $G'$ then it certifies that $\theta(Z)$ and $\theta(Z_{0})$ are not adjacent in $R(M|H)$, contrary to our earlier conclusion.
Therefore $P$ contains at least one vertex in $Y$.
Consider the maximal subpath of $P$ from $a$ to vertex not in $Y$, and let this vertex be $w$.
Note that $w$ is in $W-\theta(Z)\subseteq \theta(Z_{0})-\theta(Z)$.
So this subpath certifies that $\theta(Z)$ and $\theta(Z_{0})$ are not adjacent in $R(M|H)$, and we have another contradiction that completes the proof.
\end{proof}

\begin{proof}[Proof of \textup{Theorem \ref{ROTgraph-same-as-RCgraph}}]
Lemma \ref{RCgraph-is-ROTgraph} shows that every reduced clique graph is isomorphic to a rotunda graph.
On the other hand, if $M$ is a supersolvable saturated matroid with connected components $M_{1},\ldots, M_{n}$, then $R(M)$ is the disjoint union of $R(M_{1}),\ldots, R(M_{n})$, as we observed in Proposition \ref{conn-components}.
Lemma \ref{ROTgraph-is-RCgraph} shows that each $R(M_{i})$ is isomorphic to $C_{R}(G_{i})$ for some $2$\dash connected chordal graph $G_{i}$.
If $G$ is the disjoint union of $G_{1},\ldots, G_{n}$, then $C_{R}(G)$ is the disjoint union of $C_{R}(G_{1}),\ldots, C_{R}(G_{n})$, and is thus isomorphic to $R(M)$.
So any rotunda graph is isomorphic to a reduced clique graph.
\end{proof}

\begin{lemma}
\label{ROTgraph-connected}
Let $M$ be a supersolvable saturated matroid.
Then $R(M)$ is connected if and only if $M$ is connected.
\end{lemma}

\begin{proof}
In Proposition \ref{conn-components} we noted that if $N_{1},\ldots, N_{k}$ are the connected components of $M$, then $R(M)$ is the disjoint union of $R(N_{1}),\ldots, R(N_{k})$.
So if $M$ is not connected then neither is $R(M)$.
For the converse, we let $M$ be a connected supersolvable saturated matroid.
Lemma \ref{ROTgraph-is-RCgraph} shows that $R(M)$ is isomorphic to $C_{R}(G)$ where $G$ is a $2$\dash connected chordal graph $G$.
From Corollary 3.1 in \cite{graphpaper} we see that $C_{R}(G)$, and hence $R(M)$, is connected.
\end{proof}

\section{Clique trees and rotunda trees}
\label{TREES}

\begin{definition}
Let $M$ be a matroid and let $T$ be a tree.
Let $\tau$ be a function from $V(T)$ to $\mathcal{P}(E(M))$.
Assume that for every element $x\in E(M)$ there is at least one vertex $v\in V(T)$ such that $x\in \tau(v)$.
In this case we say that $(T,\tau)$ is a \emph{tree-decomposition} of $M$.
If for every element $x\in E(M)$ there is \emph{exactly one} vertex $v\in V(T)$ such that $x\in\tau(v)$ then the tree-decomposition is \emph{strict}.
\end{definition}

In other words, the tree-decomposition is strict if $\{\tau(t)\}_{t\in V(T)}$ is a partition of $E(M)$.

Let $G$ be a graph.
A \emph{clique tree} of $G$ is a pair $(T,\rho)$ where $T$ is a tree and $\rho$ is a bijection from $V(T)$ to the set of maximal cliques of $G$.
We insist that for any $v\in V(G)$, the set $\{t\in V(T)\colon v\in \rho(t)\}$ induces a subtree of $T$.
Clique trees were introduced by Gavril \cite{Gavril}, who showed that a graph has a clique tree if and only if it is chordal.
Our next step is to define a matroid analogue of a clique tree.

\begin{definition}
Let $M$ be a matroid, and let $(T,\tau)$ be a tree-decomposition of $M$ such that $\tau$ is a bijection from $V(T)$ to $\mathcal{R}(M)$.
If, for every $x\in E(M)$, the set $\{t\in V(T)\colon x\in \tau(t)\}$ induces a subtree of $T$, then $(T,\tau)$ is a \emph{rotunda tree} of $M$.
\end{definition}

In the following material we must apply weights to the edges of reduced clique graphs and rotunda graphs.
Let $G$ be a chordal graph.
Let $\sigma$ be a function which takes the set
\[
\{\emptyset\}\cup\{C\cap C'\colon C\ \text{and}\ C'\ \text{are distinct maximal cliques of}\ G\}
\]
to non-negative integers, and where the following conditions hold:
\begin{enumerate}[label = \textup{(\roman*)}]
\item $\sigma(\emptyset) = 0$,
\item if $X$ and $X'$ are in the domain of $\sigma$ and $X$ is a proper subset of $X'$, then $\sigma(X)<\sigma(X')$.
\end{enumerate}
In this case $\sigma$ is a \emph{legitimate weighting} of $G$.
The function $\sigma$ applies a weight to each edge of $C_{R}(G)$, where the weight of the edge between $C$ and $C'$ is $\sigma(C\cap C')$.
The following result is the main theorem of \cite{graphpaper}.

\begin{theorem}
\label{maintheorem2}
Let $G$ be a connected chordal graph and let
$\sigma$ be a legitimate weighting.
Every clique tree is a spanning tree of $C_{R}(G)$ and every edge of $C_{R}(G)$ is contained in a clique tree.
Moreover, a spanning tree of $C_{R}(G)$ is a clique tree if and only if it has maximum weight amongst all spanning trees.
\end{theorem}

Galinier, Habib, and Paul \cite{GHP95} prove the special case of Theorem \ref{maintheorem2} where $\sigma(C\cap C') = |C\cap C'|$, but their proof contains a flaw which is explained in \cite{graphpaper}.
Next we consider the matroid analogue of legitimate weightings.

\begin{definition}
Let $M$ be a supersolvable saturated matroid.
Let $\sigma$ be a function taking
\[
\{\emptyset\}\cup\{R\cap R'\colon R, R'\in \mathcal{R}(M), R\ne R'\}
\]
to non-negative integers, where:
\begin{enumerate}[label = \textup{(\roman*)}]
\item $\sigma(\emptyset) = 0$,
\item if $X$ and $X'$ are in the domain of $\sigma$ and $X$ is a proper subset of $X'$, then $\sigma(X)<\sigma(X')$.
\end{enumerate}
Then $\sigma$ is a \emph{legitimate weighting} of $M$.
\end{definition}

For examples of legitimate weightings, we may set $\sigma(R\cap R')$ to be either the rank or the size of $R\cap R'$, for each pair of rotunda $R$ and $R'$.
In the case where we use rank, the legitimacy of the weighting relies on the fact that the intersection of two rotunda is a flat.

Now we are able to prove Theorem \ref{trees_main_1}, which we restate in a more general form here.

\begin{theorem}
\label{trees_main_2}
Let $M$ be a connected supersolvable and saturated matroid and let $\sigma$ be a legitimate weighting of $M$.
Every rotunda tree of $M$ is a spanning tree of $R(M)$ and every edge of $R(M)$ is contained in a rotunda tree.
Moreover, a spanning tree of $R(M)$ is a rotunda tree if and only if it has maximum weight amongst all spanning trees.
\end{theorem}

\begin{proof}
We apply Lemma \ref{ROTgraph-is-RCgraph} and let $G$ be a $2$\dash connected chordal graph and let $\theta\colon E(M)\to\mathcal{P}(V(G))$ be a function such that $(G,\theta)$ is compliant with $M$.
Let $H$ be a graph that is isomorphic to both $C_{R}(G)$ and $R(M)$.
Let $\pi_{G}$ be a bijection from $V(H)$ to the family of maximal cliques of $G$, and let $\pi_{M}$ be a bijection from $V(H)$ to $\mathcal{R}(M)$, such that $\pi_{G}$ and $\pi_{M}$ are both isomorphisms.

Let $(T,\tau)$ be a rotunda tree of $M$.
Define $\rho$ to be the composition $\theta|_{\mathcal{R}(M)}\circ \tau$.
This means that $\rho$ is a bijection from $V(T)$ to the set of maximal cliques of $G$.
Let $v$ be an arbitrary vertex of $G$, and let $x$ be the unique element of $E(M)$ such that $v$ is in $\theta(x)$.
Now
\begin{equation}
\label{eq2}
\{t\in V(T)\colon v\in \rho(t)\} = \{t\in V(T)\colon x\in \tau (M)\}.
\end{equation}
Because the latter set induces a connected subgraph of $T$, so does the former.
This shows that $(T,\rho)$ is a clique tree of $G$.
Therefore $T$ is (isomorphic to) a spanning tree of $H$ by Theorem \ref{maintheorem2}.
We have now shown that any rotunda tree of $M$ is a spanning tree of $R(M)$.
Moreover, if $e$ is an arbitrary edge of $H$, then there is some spanning tree $T$ of $H$ such that $T$ contains $e$ and $(T,\rho)$ is a clique tree of $G$ for some bijection $\rho$.
Let $\tau$ be the composition $(\theta|_{\mathcal{R}(M)})^{-1}\circ\rho$, so that $\tau$ is a bijection from $V(T)$ to $\mathcal{R}(M)$.
If $x$ is an arbitrary element of $E(M)$ and $v$ is a vertex in $\theta(x)$, then Equation \eqref{eq2} still holds and we see that $(T,\tau)$ is a rotunda tree of $M$ that contains the edge $e$.
Thus any edge of $R(M)$ is contained in a rotunda tree of $M$.

We apply weights to the edges of $H$.
If $u$ and $u'$ are adjacent in $H$, then we weight the edge between them with $\sigma(\pi_{M}(u)\cap\pi_{M}(u'))$.
It is not difficult to see that this weighting of $H$ is also a legitimate weighting of $C_{R}(G)$;
that is, if $C$ and $C'$ are maximal cliques of $G$ that are adjacent in $C_{R}(G)$, and $\sigma_{G}$ applies the weight $\sigma(\theta^{-1}(C)\cap\theta^{-1}(C'))$ to the edge between $C$ and $C'$, then $\sigma_{G}$ is a legitimate weighting of $G$.

Let $T$ be a maximum-weight spanning tree of $H$.
Then $(T,\pi_{G})$ is a clique tree of $G$, by Theorem \ref{maintheorem2}.
Exactly as before, we see that $(T,(\theta|_{\mathcal{R}(M)})^{-1}\circ\pi_{G})$ is a rotunda tree of $M$.
On the other hand, if $T$ is a spanning tree of $H$ and $(T,\tau)$ is a rotunda tree of $M$, then $(T,\theta|_{\mathcal{R}(M)}\circ\tau)$ is a clique tree of $G$.
Hence $T$ is a maximum-weight spanning tree of $H$.
We have now proved that the rotunda trees of $M$ are exactly the maximum-weight spanning trees of $R(M)$, as claimed.
\end{proof}

It follows from Theorem \ref{trees_main_1} that $R(M)$ is the exactly the union of all rotunda trees of $M$.

\section{Tree-decompositions}
\label{DECOMPS}

We recall the definition of graph tree-width.
Let $G$ be a graph.
Let $T$ be a tree and let $\rho$ be a function from $V(T)$ to $\mathcal{P}(V(G))$ such that for every $v\in V(G)$ the set $\{t\in V(T)\colon v\in\rho(t)\}$ is non-empty and induces a subtree of $T$.
We further insist that if $u$ and $v$ are adjacent vertices of $G$, then $u,v\in\rho(t)$ for some $t\in V(T)$.
Then $(T,\rho)$ is a \emph{tree-decomposition} of $G$, and the sets $\rho(t)$ are the \emph{bags} of the decomposition.
The \emph{width} of $(T,\rho)$ is the maximum size of a bag, and the \emph{tree-width} of $G$ is the minimum width taken over all tree-decompositions.

Any clique tree of a chordal graph is a tree-decomposition of optimal width, where the bags of the tree-decomposition are exactly the maximal cliques \cite{Heggernes}*{p.~14}.
We now move towards a matroid analogue of this result.
We first introduce the notion of \emph{matroid tree-width}, as developed by Hlin\v{e}n\'{y} and Whittle \cite{hlineny_whittle}.
Recall that a tree-decomposition of a matroid $M$ is a tree $T$ along with a function $\tau\colon V(T)\to\mathcal{P}(E(M))$ such that every element $x\in E(M)$ is in at least one set $\tau(t)$.

\begin{definition}
Let $M$ be a matroid and let $(T,\tau)$ be a tree-decomposition of $M$.
Let $t$ be a node of $T$ and let $T_{1},\ldots, T_{d}$ be the connected components of $T-t$.
For each $i$ let $F_{i}$ be $\cup_{s\in V(T_{i})} \tau(s)$.
We define the \emph{node-width} of $t$ to be
\[
\left(\sum_{i=1}^{d}r\left(\tau(t)\cup \mathop{\bigcup_{k=1}^{d}}_{k\ne i}F_{k}\right)\right)-(d-1)r(M).
\]
The \emph{width} of $(T,\tau)$ is the maximum node-width of any node in $T$.
The \emph{tree-width} of $M$ (denoted $\tw(M)$) is the smallest width of any tree-decomposition of $M$.
\end{definition}

Note that this definition is not exactly that used by Hlin\v{e}n\'{y} and Whittle because in their definition the minimum ranges over \emph{strict} tree-decompositions, rather than all tree-decompositions.
To see that this makes no difference to the definition, assume that the element $x\in E(M)$ is contained in both $\tau(u)$ and $\tau(v)$, where $u$ and $v$ are distinct vertices of the tree $T$.
We redefine $\tau$ by removing $x$ from $\tau(u)$.
It is easy to confirm that the width of no node is increased by this change.
By repeating this process we can produce a strict tree-decomposition with width no greater than the width of our original decomposition.
This argument shows that there exists a strict tree-decomposition whose width is as small as possible amongst all tree-decompositions.
Thus extending Hlin\v{e}n\'{y} and Whittle's definition to include non-strict tree-decompositions makes no difference to the parameter.

We can always let $T$ be a tree with a single node, and let $\tau$ take every element of $E(M)$ to this node.
It follows from the definition that the width of $(T,\tau)$ is $r(M)$.
This shows that the tree-width of any matroid $M$ is bounded above by $r(M)$.

\begin{proposition}
\label{tw-round}
Let $M$ be a round matroid.
Then $\tw(M) = r(M)$.
\end{proposition}

\begin{proof}
Let $E$ be the ground set of $M$.
Let $(T,\tau)$ be any strict tree-decomposition of $M$.
We direct each edge of $T$ in the following way.
Let $e$ be an arbitrary edge of $T$ and assume that $e$ joins $u_{1}$ to $u_{2}$.
For each $i$ let $T_{i}$ be the connected component of $T\backslash e$ that contains $u_{i}$.
Let $U_{i}=\cup_{s\in V(T_{i})}\tau(s)$.
Thus $(U_{1},U_{2})$ is a partition of $E$ (since the tree-decomposition is strict), and because $M$ is round, either $U_{1}$ or $U_{2}$ is spanning.
If $U_{i}$ is spanning then we direct $e$ from $u_{3-i}$ to $u_{i}$.
Note that it is possible for an edge to have two directions applied to it.

Let $P$ be a maximum length directed path in $T$, and assume that $t$ is the final node in $P$.
Let $T_{1},\ldots, T_{d}$ be the connected components of $T-t$ and let $F_{i}=\cup_{s\in V(T_{i})}\tau(s)$.
Because the edges incident with $t$ are all directed towards $t$, it follows that $E-F_{i}$ is spanning for each $i$.
Since $F_{1},\ldots, F_{d}$ are pairwise disjoint, the width of $t$ is
\[
r(M)-\sum_{i=1}^{d}(r(M)-r(E-F_{i})) = r(M) - \sum_{i=1}^{d}(r(M)-r(M)) = r(M).
\]
Hence the node-width of $t$ is equal to $r(M)$.
Thus $\tw(M) \geq r(M)$.
We have already observed that $\tw(M)\leq r(M)$ so  the proof is complete.
\end{proof}

Hlin\v{e}n\'{y} and Whittle show that if $N$ is a minor of the matroid $M$, then $\tw(N)\leq \tw(M)$ \cite{hlineny_whittle}*{Proposition 3.1}.
The next result follows from this observation and Proposition \ref{tw-round}.

\begin{corollary}
\label{tw-rotunda}
Let $M$ be a matroid and let $R$ be a round flat of $M$.
Then $\tw(M) \geq r(R)$.
\end{corollary}

\begin{proposition}
\label{modular-edge}
Let  $(T,\tau)$ be a rotunda tree of $M$, a supersolvable saturated matroid.
Let $e$ be an edge of $T$ that joins vertices $u_{1}$ and $u_{2}$.
For $i=1,2$, let $T_{i}$ be the connected component of $T\backslash e$ that contains $u_{i}$ and let $F_{i}$ be $\cup_{t\in V(T_{i})}\tau(t)$.
Then $(F_{1},F_{2})$ is a modular cover of $M$ and $F_{1}\cap F_{2} = \tau(u_{1})\cap\tau(u_{2})$.
\end{proposition}

\begin{proof}
Note that every element of $E(M)$ is contained in a round flat, and hence in a rotunda.
From this it follows that $E(M)=F_{1}\cup F_{2}$.

We apply Lemma \ref{ROTgraph-is-RCgraph} and we let $G$ be a $2$\dash connected chordal graph with a function $\theta\colon E(M)\to\mathcal{P}(V(G))$ such that $(G,\theta)$ is compliant with $M$.
Let $\rho$ be the composition $\theta|_{\mathcal{R}(M)}\circ\tau$ so that $\rho$ is a bijection between $V(T)$ and the maximal cliques of $G$.
Exactly as in the proof of Theorem \ref{trees_main_1} we can show that $(T,\rho)$ is a clique tree of $G$.

Define $R_{i}$ to be the rotunda $\tau(u_{i})$.
Let $F$ be the flat $R_{1}\cap R_{2}$.
Note that because $R_{1}$ and $R_{2}$ are adjacent in a rotunda tree of $M$, they are adjacent in $R(M)$ by Theorem \ref{trees_main_1}.
This implies that $F$ is non-empty.
Let $C_{i}=\theta(R_{i})$ for $i=1,2$, so that $C_{1}$ and $C_{2}$ are the corresponding maximal cliques of $G$.
Define $S$ to be $\theta(F) = C_{1}\cap C_{2}$.

Note that if $D$ is a maximal clique of $G$, then $D-S$ is contained in a connected component of $G-S$.
For $i=1,2$, let $v_{i}$ be an arbitrary vertex of $T_{i}$.
Then the path of $T$ from $v_{1}$ to $v_{2}$ contains $u_{1}$ and $u_{2}$.
It follows from \cite{graphpaper}*{Proposition 2.8} that $\rho(v_{1})-S$ and $\rho(v_{2})-S$ are contained in different connected components of $G-S$.
Now we let $U$ be the union of all connected components of $G-S$ that contains $\rho(v)-S$ for some $v$ in $V(T_{1})$.
From the observations in this paragraph we see that $F\cup\theta^{-1}(U)$ is equal to $F_{1}$.
Because $(G,\theta)$ is compliant with $M$ this means that $F_{1}$ is a modular flat of $M$.
Symmetrically, $F_{2}$ is a modular flat.

Let $x$ be an arbitrary element in $F_{1}\cap F_{2}$.
Let $v\in V(T_{1})$ and $v'\in V(T_{2})$ be chosen so that $x$ is in $\tau(v)\cap\tau(v')$.
Because $(T,\tau)$ is a rotunda tree it follows that $x$ is in $\tau(w)$ whenever $w$ is in the path of $T$ from $v$ to $v'$.
In particular, $x$ is in $\tau(u)\cap\tau(u') = R\cap R' = F$.
Thus $F_{1}\cap F_{2}\subseteq F$.
Because $u_{i}$ is in $T_{i}$ for each $i$ it follows that $R_{i}\subseteq F_{i}$.
Therefore $F=R_{1}\cap R_{2}$ is a subset of $F_{1}\cap F_{2}$, and now
\[\tau(u_{1})\cap\tau(u_{2})=R_{1}\cap R_{2} = F = F_{1}\cap F_{2}.\]
From this it follows that $F_{1}\cap F_{2}$ does not contain $R_{1}$ or $R_{2}$, so neither $F_{1}$ nor $F_{2}$ is equal to $E(M)$.
Since $F_{1}$ and $F_{2}$ are proper modular flats of $M$ and $E(M)=F_{1}\cup F_{2}$ we see that $(F_{1},F_{2})$ is a modular cover and the result is proved.
\end{proof}

Let $M$ be a connected supersolvable and saturated matroid.
We will now show that a rotunda tree of $M$ has the properties of an optimal tree-decomposition as per Hlin\v{e}n\'{y} and Whittle.

\begin{theorem}
\label{tree_theorem_two}
Let $M$ be a supersolvable saturated matroid and let $(T,\tau)$ be a rotunda tree of $M$.
Then the width of $(T,\tau)$ is equal to $\tw(M)$.
\end{theorem}

\begin{proof}
We will show that the node-width of any $t\in V(T)$ is $r(\tau(t))$, so that the width of $(T,\tau)$ is the maximum rank of a rotunda of $M$.
From Corollary \ref{tw-rotunda} we see that $\tw(M)$ is bounded below by this rank, so having completed this task, we will have shown that $(T,\tau)$ is a tree-decomposition of lowest-possible rank.
It will then follow that $\tw(M)$ is equal to the width of $(T,\tau)$.

So let $t$ be an arbitrary vertex in $T$ and let $T_{1},\ldots, T_{d}$ be the connected components of $T-t$.
For each $i$ let $t_{i}$ be the vertex of $T_{i}$ that is adjacent to $t$.
Define $F$ to be $\tau(t)$, and let $F_{i}$ be $\cup_{s\in V(T_{i})}\tau(s)$ for each $i$.
We define $\overline{F}_{i}$ to be
\[
F\cup \mathop{\bigcup_{k=1}^{d}}_{k\ne i}F_{k}.
\]
Therefore the node-width of $t$ is
\begin{equation}
\label{eq3}
r(\overline{F}_{1})+\cdots+r(\overline{F}_{d})-(d-1)r(M).
\end{equation}
In addition, we define $\overline{F}_{>i}$ to be
\[
F\cup \bigcup_{k=i+1}^{d}F_{k}.
\]
Notice that $\overline{F}_{> 1} = \overline{F}_{1}$ and that $\overline{F}_{>d}=F$.

\begin{claim}
\label{tree-claim1}
For any $i\in\{1,\ldots, d-1\}$, the intersection of $\overline{F}_{>i}$ and $\overline{F}_{i+1}$ is $\overline{F}_{>i+1}$.
\end{claim}

\begin{proof}
We note that
\begin{multline*}
\label{eq4}
\overline{F}_{>i}\cap \overline{F}_{i+1}
=(F\cup F_{i+1}\cup \cdots \cup F_{d})\cap (F\cup F_{1}\cup\cdots \cup F_{i}\cup F_{i+2}\cup\cdots F_{d})\\
=(F_{i+1}\cap(F_{1}\cup\cdots\cup F_{i}))\cup(F\cup F_{i+2}\cup\cdots\cup F_{d}).
\end{multline*}
Now $F_{i+1}\cap(F_{1}\cup\cdots\cup F_{i})$ is contained in $F_{i+1}\cap \overline{F}_{i+1}$.
But Proposition \ref{modular-edge} tells us that $F_{i+1}\cap \overline{F}_{i+1}$ is equal to $\tau(t_{i+1})\cap \tau(t)$, which is therefore contained in $\tau(t)=F$.
Hence we can remove $F_{i+1}\cap(F_{1}\cup\cdots\cup F_{i})$ from the equation above and conclude that $\overline{F}_{>i}\cap\overline{F}_{i+1}$ is $F\cup F_{i+2}\cup\cdots\cup F_{d} = \overline{F}_{>i+1}$, as claimed.
\end{proof}

Proposition \ref{modular-edge} implies that $(F_{i},\overline{F}_{i})$ is a modular cover for each $i$, so that in particular $\overline{F}_{2}$ is a modular flat.
Now Equation \eqref{eq3} reduces to
\begin{align*}
&r(\overline{F}_{1}\cap \overline{F}_{2})+
r(\overline{F}_{1}\cup \overline{F}_{2})+r(\overline{F}_{3})+\cdots+r(\overline{F}_{d})-(d-1)r(M)\\
&=r(\overline{F}_{>1}\cap \overline{F}_{2})+r(E(M))
+r(\overline{F}_{3})+\cdots+r(\overline{F}_{d})-(d-1)r(M)\\
&=r(\overline{F}_{>2})
+r(\overline{F}_{3})+\cdots+r(\overline{F}_{d})-(d-2)r(M)
\end{align*}
where we have applied \ref{tree-claim1} in the final step.
Because $\overline{F}_{3}$ is a modular flat, we can again apply \ref{tree-claim1} and reduce to
\[
r(\overline{F}_{>3})
+r(\overline{F}_{4})+\cdots+r(\overline{F}_{d})-(d-3)r(M)
\]
By continuing this process, we find that Equation \eqref{eq3} is equal to
\[
r(\overline{F}_{>d})-(d-d)r(M) = r(F).
\]
So the node-width of $t$ is $r(\tau(t))=r(F)$, exactly as we claimed, and the theorem is proved.
\end{proof}

Now we present the central theorem for this section.
Because a rotunda tree is a tree-decomposition of optimal width for a supersolvable saturated matroid $M$, we can treat it as a canonical tree decomposition of $M$.

\begin{corollary}
\label{bound_on_rotunda_rank}
Let $M$ be a supersolvable saturated matroid. Then $\tw(M)=\max\{r(R)\colon R\in\mathcal{R}(M)\}$.
\end{corollary}

Further observe the following.
Let $\bw(M)$ denote the branch-width of $M$.
By \cite{hlineny_whittle}*{Theorem 4.2} we see that
\[
\bw(M) - 1 \leq \tw(M) \leq \max\{2\bw(M) - 2, 1\}.
\]
We see therefore that given a supersolvable saturated matroid $M$ of branch-width $k$ there must be a rotunda tree of $M$ where the rank of the largest maximal rotunda is bounded by a function of $k$.
As a result, we can conclude that supersolvable saturated matroids have canonical tree decompositions of optimal tree-width in much the same way as chordal graphs have canonical tree decompositions where each bag is a clique of the graph.

This theorem has algorithmic implications for how we can efficiently find the tree-width of a supersolvable saturated matroid.
However, for this to work we would need an efficient method for constructing the rotunda graph.

\section{Acknowledgements}

We thank Geoff Whittle, who supervised the thesis of the second author (which includes much of the material in this article).
We also thank a referee of an earlier draft for numerous helpful comments.

\end{document}